\documentclass[reqno,12pt]{amsart}
\usepackage[margin=1.1in, footskip=1cm]{geometry}
\usepackage{amsmath, amsthm, amssymb, booktabs}
\usepackage{mathrsfs}
\pdfpagewidth=\paperwidth
\pdfpageheight=\paperheight

\newtheorem{theorem}{Theorem}[section]
\newtheorem{lemma}[theorem]{Lemma}
\newtheorem{corollary}[theorem]{Corollary}

\newtheorem{problem}[theorem]{Problem}

\theoremstyle{definition}
\newtheorem{definition}[theorem]{Definition}
\newtheorem{example}[theorem]{Example}

\theoremstyle{remark}

\numberwithin{equation}{section}

\newcommand{\mmod}[1]{\,\,({\rm{mod}}\,\,#1)}

\def\bfh{{\mathbf h}}

\def\bfm{{\mathbf m}}

\def\grm{{\mathfrak m}}

\def\bet{{\beta}}  
 
 \def\Del{{\Delta}}

\def\d{{\partial}}

\def\le{\leqslant} \def\ge{\geqslant}

\def\d{{\,{\rm d}}}

\makeatletter
\@namedef{subjclassname@2020}{\textup{2020} Mathematics Subject Classification}
\makeatother

\begin{document}
\title[Subconvex $L^p$-sets]{Subconvex $L^p$-sets, Weyl's inequality,\\ and equidistribution}
\author[Trevor D. Wooley]{Trevor D. Wooley}
\address{Department of Mathematics, Purdue University, 150 N. University Street, West 
Lafayette, IN 47907-2067, USA}
\email{twooley@purdue.edu}
\thanks{Work on this paper was conducted while the author was supported in part by NSF grants DMS-2001549 
and DMS-2502625, and Simons Fellowship in Mathematics SFM-00011955.}
\date{}

\begin{abstract} We examine sets $\mathscr A$ of natural numbers having the property that for some real number 
$p\in (0,2)$, one has the subconvex bound
\[
\int_0^1 \Bigl| \sum_{n\in \mathscr A\cap [1,N]}e(n\alpha)\Bigr|^p\d \alpha \ll N^{-1}|\mathscr A\cap [1,N]|^p.
\]
We show that exponential sums over such sets satisfy inequalities analogous to Weyl's inequality, and in many 
circumstances of the same strength as classical versions of Weyl's bound. We also examine equidistribution of 
polynomials modulo $1$ in which the summands are restricted to these subconvex $L^p$-sets. In addition, we 
describe applications to problems involving character sums and averages of arithmetic functions.
\end{abstract}

\maketitle

\section{Introduction}
The inequality of Weyl provides non-trivial estimates for exponential sums having real polynomial arguments 
summed over all integers of an interval. In this memoir, we show that certain subsets of the natural numbers, that 
we term {\it subconvex $L^p$-sets}, satisfy analogues of Weyl's inequality losing nothing by comparison with the 
strength of these classical bounds. These bounds will facilitate applications of the Hardy-Littlewood (circle) method 
to problems involving subconvex $L^p$-sets. We illustrate our ideas by considering equidistribution results for 
sequences of polynomials modulo $1$, and averages of arithmetic functions, in which the underlying arguments are 
restricted to subconvex $L^p$-sets. In view of our new conclusions, the investigation and characterisation of 
subconvex $L^p$-sets are tasks that seem worthy of future attention.\par

We now introduce the sets central to our discussion. Throughout this memoir, when 
$\mathscr A\subseteq \mathbb N$ and $N\ge 1$, we write $\mathscr A(N)=\mathscr A\cap [1,N]$ and 
$A(N)=\text{card}(\mathscr A(N))$. As usual, we write $e(z)$ for $e^{2\pi {\rm i}z}$. Then, when $p$ is a positive real 
number, we define the mean value
\begin{equation}\label{1.1}
I_p(N;\mathscr A)=\int_0^1 \Bigl| \sum_{n\in \mathscr A(N)}e(n\alpha)\Bigr|^p\d \alpha .
\end{equation}

\begin{definition}\label{definition1.1} Suppose that $\mathscr A$ is a non-empty subset of $\mathbb N$.
\begin{enumerate}
\item[(a)] We say that $\mathscr A$ is a {\it weakly subconvex $L^p$-set} if $0<p<2$ and, for all $\varepsilon >0$ 
and all real numbers $N$ sufficiently large in terms of $p$ and $\varepsilon$, one has
\begin{equation}\label{1.2}
I_p(N;\mathscr A)\ll N^{\varepsilon -1}A(N)^p.
\end{equation}
\item[(b)] We say that $\mathscr A$ is a {\it strongly subconvex $L^p$-set} if $0<p<2$ and, for all real numbers 
$N$ sufficiently large in terms of $p$, one has
\begin{equation}\label{1.3}
I_p(N;\mathscr A)\ll N^{-1}A(N)^p.
\end{equation}
\end{enumerate}
\end{definition}

We make three simple observations concerning this definition. First, we note that any strongly subconvex $L^p$-set 
$\mathscr A$ has positive lower density. For if $0<p<2$, it follows from \eqref{1.1} via orthogonality, a trivial 
estimate, and the upper bound \eqref{1.3} that
\[
A(N)=I_2(N;\mathscr A)\le A(N)^{2-p}I_p(N;\mathscr A)\ll N^{-1}A(N)^2.
\]
Thus, we infer that $A(N)\gg N$. The lower bound $A(N)\gg N^{1-\varepsilon}$ follows in like manner from 
\eqref{1.2} for weakly subconvex $L^p$-sets $\mathscr A$.\par

Second, our use of the adjective {\it subconvex} is designed to remind the reader that the estimate \eqref{1.3} 
constitutes a bound beyond square-root cancellation, a claim justified by the observation that strongly subconvex 
$L^p$-sets have positive lower density. Indeed, were the exponential sum associated with $\mathscr A$ to exhibit 
only square-root cancellation for a set of arguments $\alpha$ inside the integral \eqref{1.1} having positive 
measure, we would have $I_p(N;\mathscr A)\gg A(N)^{p/2}\gg N^{p/2}$. Since $p/2>p-1$ when $0<p<2$, it is 
apparent that the upper bound \eqref{1.3} is in this sense subconvex. Similar observations apply, mutatis mutandis, 
to weakly subconvex $L^p$-sets.\par

Third, it is apparent that when $|\alpha|\le 1/(100N)$, one has
\[
\biggl| \sum_{n\in \mathscr A(N)}e(n\alpha)\biggr|>\tfrac{1}{2}A(N).
\]
Thus, we deduce from \eqref{1.1} that for $0<p<2$, one has $I_p(N;\mathscr A)\gg N^{-1}A(N)^p$. The 
hypothesised upper bound \eqref{1.3} is consequently the best achievable for a $p$-th moment. In this context, we 
note that when $0<p<1$, the upper bound \eqref{1.3} is irritatingly strong, since then one has $I_p(N)=o(1)$ as 
$N\rightarrow \infty$. We stress that this scenario is included within the definition solely for the sake of 
completeness. Indeed, it is a consequence of the resolution of Littlewood's problem (see Konyagin \cite{Kon1981} and 
McGehee, Pigno and Smith \cite[Corollary 1]{MPS1981}) that $I_1(N;\mathscr A)\gg \log A(N)$. Thus, if it were to be the case that 
$\mathscr A$ is a strongly subconvex $L^p$-set with $0<p\le 1$, then one would have
\[
\log A(N)\ll I_1(N;\mathscr A)\le A(N)^{1-p}I_p(N;\mathscr A)\ll N^{-1}A(N),
\]
whence $A(N)\gg N\log A(N)$. Since we may assume that $\mathscr A$ has positive lower density, we obtain a contradiction. Thus, there exist no strongly 
subconvex $L^p$-sets when $0<p\le 1$. Our focus in this memoir lies with those situations in which $1\le p<2$. It is this latter range of $p$ 
for which we describe applications herein, and for which we envision the development of a coherent theory. The 
reader will lose nothing by assuming henceforth that $p$ is constrained to lie in this range.\par

It might seem that subconvex $L^p$-sets, whether weak or strong, are subsets of the integers having such special 
properties that they should be difficult to find, and might not perhaps occur in nature. As we explain in \S\S2 and 3, 
however, examples of subconvex $L^p$-sets are abundant. The set of all natural numbers $\mathbb N$ is a trivial 
example of a weakly subconvex $L^1$-set which is also a strongly subconvex $L^p$-set whenever $p>1$. This 
follows from the well-known bounds
\[
\int_0^1 \Bigl| \sum_{1\le n\le N}e(n\alpha)\Bigr| \d\alpha \asymp \log (2N) \quad \text{and}\quad 
\int_0^1 \Bigl| \sum_{1\le n\le N}e(n\alpha)\Bigr|^p\d\alpha \ll N^{p-1}\quad (p>1).
\]
We remind the reader of these familiar bounds, and their proofs, in Theorem \ref{theorem2.1} and the ensuing 
discussion. Consider next, when $r\ge 2$, the set $\mathscr N_r$ of $r$-free numbers, defined by
\[
\mathscr N_r=\{ n\in \mathbb N: \text{$\pi^r|n$ for no prime $\pi$}\}.
\]
It transpires that $\mathscr N_r$ is a weakly subconvex $L^{1+1/r}$-set, and when $p>1+1/r$ it is a strongly 
subconvex $L^p$-set. Indeed, if we write $\mu_r(n)$ for the characteristic function of $\mathscr  N_r$, then these 
assertions follow from the work of Keil \cite[Theorem 1.2]{Kei2013} showing that
\begin{equation}\label{1.4}
N^{1/r}\log (2N)\ll \int_0^1\Bigl| \sum_{1\le n\le N}\mu_r(n)e(n\alpha)\Bigr|^{1+1/r}\d\alpha \ll N^{1/r}\log^2(2N)
\end{equation}
and
\begin{equation}\label{1.5}
\int_0^1\Bigl| \sum_{1\le n\le N}\mu_r(n)e(n\alpha)\Bigr|^p\d\alpha \asymp N^{p-1}\quad (p>1+1/r).
\end{equation}
In \S2 we exhibit other examples of subconvex $L^p$-sets, and in \S3 we show how new examples of subconvex 
$L^p$-sets may be derived from known examples.\par

We turn in \S\S4 and 5 to the topic of Weyl sums. When $k\ge 2$ and $\alpha_i\in \mathbb R$ $(0\le i\le k)$, 
consider the polynomial
\begin{equation}\label{1.6}
\psi(x;\boldsymbol \alpha)=\alpha_kx^k+\ldots +\alpha_1x+\alpha_0,
\end{equation}
and define the exponential sum
\begin{equation}\label{1.7}
\Psi_k(\boldsymbol \alpha ;N)=\sum_{1\le n\le N}e(\psi(n;\boldsymbol \alpha)).
\end{equation}
Suppose that the leading coefficient $\alpha_k$ of $\psi(x;\boldsymbol \alpha)$ has the kind of Diophantine 
approximation made available via Dirichlet's theorem, so that $a\in \mathbb Z$ and $q\in \mathbb N$ satisfy 
$(a,q)=1$ and $|\alpha_k-a/q|\le 1/q^2$. In these circumstances, it follows from Weyl's inequality (see 
\cite[Lemma 2.4]{Vau1997}, for example) that for each $\varepsilon>0$ and each large real number $N$, one has
\begin{equation}\label{1.8}
\Psi_k(\boldsymbol \alpha;N)\ll N^{1+\varepsilon }(q^{-1}+N^{-1}+qN^{-k})^{2^{1-k}}.
\end{equation}
When the variables $n$ defining the summation in \eqref{1.7} are restricted to lie in arithmetically interesting 
subsequences of the integers, such as the prime numbers or squarefree numbers, the strength of available 
estimates analogous to \eqref{1.8} typically degrades substantially. Our first result on Weyl's inequality shows that 
no significant degradation need be tolerated for such subsequences as those defined by weakly subconvex 
$L^p$-sets, at least when $k\ge 3$ and $1\le p\le 4/3$.

\begin{theorem}\label{theorem1.2}
Suppose that $\mathscr A$ is a weakly subconvex $L^p$-set for some real number $p$ with $1\le p\le 4/3$, and 
$k\ge 3$. Let $(\alpha_0,\alpha_1,\ldots ,\alpha_k)\in \mathbb R^{k+1}$, and suppose that $a\in \mathbb Z$ and 
$q\in \mathbb N$ satisfy $(a,q)=1$ and $|\alpha_k-a/q|\le 1/q^2$. Then, for each $\varepsilon>0$ and each large 
real number $N$, one has
\begin{equation}\label{1.9}
\sum_{n\in \mathscr A(N)}e(\alpha_kn^k+\ldots +\alpha_1n+\alpha_0)\ll 
N^{1+\varepsilon}(q^{-1}+N^{-1}+qN^{-k})^{2^{1-k}}.
\end{equation}
\end{theorem}

A comparison of \eqref{1.8} and \eqref{1.9} will assure the reader that no degradation has occurred here 
accompanying our restriction of the set of summands to the subconvex $L^p$-set $\mathscr A$. Theorem 
\ref{theorem1.2} is in fact a corollary of a more general conclusion applicable also to the situations in which 
$k\ge 2$ and $\mathscr A$ is any subconvex $L^p$-set with $1\le p<2$. In order to describe this result, we 
introduce the exponents
\[
\sigma_p(k)=\begin{cases} \frac{1}{p}-\frac{1}{2},&\text{when $k=2$,}\\
2^{3-k}\bigl( \frac{1}{p}-\frac{1}{2}\bigr),&\text{when $k\ge 3$ and $4/3<p<2$,}\\
2^{1-k},&\text{when $k\ge 3$ and $1\le p\le 4/3$,}
\end{cases}
\]
and
\[
\tau_p(k)=\begin{cases}\bigl( \frac{2}{p}-1\bigr) \frac{1}{k^2-k-2},&\text{when $k\ge 3$ and 
$\frac{k^2-k}{k^2-k-1}<p<2$,}\\
\frac{1}{k(k-1)},&\text{when $k\ge 3$ and $1\le p\le \frac{k^2-k}{k^2-k-1}$,}\\
0,&\text{when $k=2$.}\end{cases}
\]

\begin{theorem}\label{theorem1.3}
Suppose that $\mathscr A$ is a weakly subconvex $L^p$-set for some real number $p$ with $1\le p<2$, and 
$k\ge 2$. Let $(\alpha_0,\alpha_1,\ldots ,\alpha_k)\in \mathbb R^{k+1}$, and suppose that $a\in \mathbb Z$ and 
$q\in \mathbb N$  satisfy $(a,q)=1$ and $|\alpha_k-a/q|\le 1/q^2$. Then, for each $\varepsilon>0$ and each large 
real number $N$, one has
\[
\sum_{n\in \mathscr A(N)}e(\alpha_kn^k+\ldots +\alpha_1n+\alpha_0)\ll 
N^{1+\varepsilon}(q^{-1}+N^{-1}+qN^{-k})^{\omega_p(k)},
\]
where $\omega_p(k)=\max\{ \sigma_p(k),\tau_p(k)\}$.
\end{theorem}

We observe that when $k\ge 3$ and $1\le p\le 4/3$, then the conclusion of Theorem \ref{theorem1.3} matches 
Weyl's inequality \eqref{1.8} in strength. Moreover, when $k\ge 3$ and
\[
1\le p\le1+\frac{1}{k^2-k-1},
\]
this theorem matches in strength the analogue of Weyl's inequality deriving from modern versions of Vinogradov's 
mean value theorem (obtained, for example, by substituting the main conclusions of \cite{BDG2016} or 
\cite{Woo2019} into \cite[Theorem 5.2]{Vau1997}). Meanwhile, even for $k=2$, the strength of Theorem 
\ref{theorem1.3} approaches that of Weyl's inequality in the limit as $p\rightarrow 1+$.\par

The focus of Theorems \ref{theorem1.2} and \ref{theorem1.3} is the derivation of exponential sum estimates with 
variables restricted to special subsequences of the natural numbers. If, instead, one is concerned solely with 
applications of the exponential sums, then one may attempt to incorporate the restriction to special subsequences 
into the methods employed following the application of conventional estimates for exponential sums. Thus, for 
example, the reader will find alternative approaches to the one presented in this memoir that extract conclusions in 
applications, avoiding impairment, when one restricts the underlying variables to be square-free, or more generally 
$r$-free. We refer the reader to \cite[Theorem 2]{BBH1991}, for example, for a conclusion concerning small values 
of $\|\alpha n^k\|$, when $n$ is restricted to be square-free. In this latter work, one avoids direct use of the exponential 
sum
\[
\sum_{1\le n\le N}\mu_2(n)e(\alpha n^k)
\]
by applying the convolution identity
\[
\mu_2(n)=\sum_{d^2|n}\mu(d),
\]
and removing the contribution of the large square factors directly. This approach achieves a conclusion for the 
application at hand matching its unrestricted analogue at the cost of various complications in the associated 
argument. The ideas underlying the proof of Theorems \ref{theorem1.2} and \ref{theorem1.3} permit such 
conclusions to be obtained in wider generality than hitherto, and oftentimes with greater economy of effort.\par

As the next application of the ideas in this memoir, in \S6, we explore some consequences for the equidistribution 
modulo $1$ of polynomial sequences. Consider a real sequence $(s_n)_{n=1}^\infty$ and the associated fractional 
parts $\{s_n\}=s_n-\lfloor s_n\rfloor$. This sequence is said to be {\it equidistributed modulo $1$} when, for each 
pair of real numbers $a$ and $b$ with $0\le a<b\le 1$, one has
\[
\lim_{N\rightarrow \infty} N^{-1}\text{card}\{ 1\le n\le N: a\le \{s_n\}\le b\}=b-a.
\]

\begin{theorem}\label{theorem1.4}
Suppose that $\mathscr A=\{a_1,a_2,\ldots \}$, with $a_1<a_2<\ldots $, is a strongly subconvex $L^p$-set with 
$1\le p<2$. Let $k\ge 2$, suppose that $(\alpha_0,\alpha_1,\ldots ,\alpha_k)\in \mathbb R^{k+1}$, and define the 
polynomial $\psi(x;\boldsymbol \alpha)$ as in \eqref{1.6}. Then, provided that one at least of the coefficients 
$\alpha_2,\ldots ,\alpha_k$ is irrational, the sequence $\left( \psi(a_n;\boldsymbol \alpha)\right)_{n=1}^\infty$ is 
equidistributed modulo $1$.
\end{theorem}

In Theorem \ref{theorem6.1} below, we provide a similar conclusion applicable to weakly subconvex $L^p$-sets, 
subject to the condition that one at least of the coefficients $\alpha_2,\ldots ,\alpha_k$ is of finite Diophantine type. 
We direct the reader to the discussion following Example \ref{example2.2} for an explanation of what it means for a 
real number $\theta$ to be of finite Diophantine type. For now, it suffices to remark that the set of such numbers 
has full measure in the real numbers, and includes all irrational real algebraic numbers. We remark also that the 
conclusion of Theorem \ref{theorem1.4} does not remain valid if one insists only that one at least of the coefficients 
$\alpha_1,\ldots ,\alpha_k$ is irrational. An example of a strongly subconvex $L^p$-set is given in the discussion  
following the proof of Theorem \ref{theorem1.4} in \S6 illustrating that equidistribution modulo $1$ may fail for the 
sequence $(\psi(a_n;\boldsymbol \alpha))_{n=1}^\infty$ when $\alpha_1$ is irrational.

In \S7, we offer some remarks concerning further applications of subconvex $L^p$-sets $\mathscr A$ to problems 
involving the estimation of averages of arithmetic functions restricted to $\mathscr A$. We direct the reader to \S7 
for a more comprehensive account of such matters. Here, the reader will find a general theorem concerning such 
averages, as well as illustrative examples involving the von Mangoldt function, character sums, and averages of cusp 
form coefficients, all restricted to subconvex $L^p$-sets. For now, we present but one of many examples, involving 
the M\"obius function $\mu (n)$.

\begin{theorem}\label{theorem1.5}
Suppose that $\mathscr A$ is a strongly subconvex $L^p$-set for some real number $p$ with $1\le p<2$. Then, 
whenever $A>0$ and $N$ is sufficiently large in terms of $A$, one has
\[
\sum_{n\in \mathscr A(N)}\mu(n)\ll N(\log N)^{-A}.
\]
\end{theorem}

By applying the results of \S3, for example, one finds that the hypotheses of this theorem are applicable when
\[
\mathscr A=\{ m+1:\text{$m$ is squarefree}\},
\]
or even
\[
\mathscr A=\{ m\in \mathbb N:\text{$m+1$ is not $4$-free and $m+2$ is $4$-free}\}.
\]

\par Finally, in \S8, we record some connections between subconvex $L^p$-sets and mean values associated with 
congruences modulo $q$. Our principal result in this setting is an analogue of the mean value \eqref{1.3} for 
strongly subconvex $L^p$-sets $\mathscr A$.

\begin{theorem}\label{theorem1.6}
Suppose that $\mathscr A$ is a strongly subconvex $L^p$-set for some real number $p$ with $1<p<2$. Then, 
whenever $q$ is a natural number, one has
\[
\frac{1}{q}\sum_{a=1}^q\biggl| \sum_{n\in \mathscr A(N)}e(na/q)\biggr|^p\ll N^{p-1}+N^p/q.
\]
\end{theorem}

When $q$ is smaller than $N$, the estimate provided by this theorem saves roughly a factor $q$ over the trivial 
estimate $N^p$, which is consistent with heuristics associated with associated congruences modulo $q$. Such 
heuristics would be easy to justify were we to be considering $p$-th moments with $p\ge 2$, since then 
orthogonality may be brought into play. The situation in Theorem \ref{theorem1.6} with $1<p<2$ is, unfortunately, 
rather more challenging. Thus, in order to establish this theorem, we are forced to bring into play purely analytic 
tools based on the use of the Sobolev-Gallagher inequality and the Carleson-Hunt theorem. We offer an application 
of Theorem \ref{theorem1.6} to character sum estimates restricted to subconvex $L^p$-sets in \S8.\par

We defer a detailed account of our use of subconvex $L^p$-sets in our methods deriving analogues of Weyl's 
inequality to \S4 below. For now, it suffices to make some abstract remarks that may offer some guidance to 
the reader, motivating our definition of subconvex $L^p$-sets. Suppose then that $\mathscr A$ is a strongly 
subconvex $L^p$-set for some real number $p$ with $1\le p\le 4/3$, and write
\[
g(\alpha)=\sum_{n\in \mathscr A(N)}e(n\alpha).
\]
Consider a unimodular arithmetic function $c:\mathbb N\rightarrow \mathbb C$, and write
\[
h(\alpha)=\sum_{1\le n\le N}c(n)e(n\alpha).
\]
Then it follows by orthogonality that
\[
\sum_{n\in \mathcal A(N)}c(n)=\int_0^1g(\alpha)h(-\alpha)\d\alpha .
\]
Thus, by applying H\"older's inequality and the Definition \ref{definition1.1}(b) of a strongly subconvex $L^p$-set, 
we find that
\begin{align*}
\Bigl| \sum_{n\in \mathscr A(N)}c(n)\Bigr|&\le \biggl( \int_0^1 |g(\alpha)|^{4/3}\d\alpha \biggr)^{3/4}
\biggl( \int_0^1 |h(\alpha)|^4\d\alpha \biggr)^{1/4}\\
&\ll (N^{1/3})^{3/4}\biggl( \sum_{\substack{1\le n_1,\ldots ,n_4\le N\\ n_1+n_2=n_3+n_4}}c(n_1)c(n_2)
{\overline{c(n_3)}}{\overline{c(n_4)}}\biggr)^{1/4}.
\end{align*}
The $4$-fold sum in the final parenthetic term may often be interpreted via Weyl differencing, or perhaps by other 
means. Thus, it may frequently be shown that this sum is of smaller order than the trivial estimate $N^3$. In such 
circumstances, we conclude that one has the non-trivial estimate
\[
\biggl| \sum_{n\in \mathscr A(N)}c(n)\biggr| =o(N).
\]
We emphasise here that the arithmetic structure of the set $\mathscr A$ makes no appearance in this estimate 
beyond the estimate \eqref{1.3} that is the fundamental property of a strongly subconvex $L^p$-set.\par

We finish by highlighting some connections between the subconvex $L^p$-sets of this memoir, and the almost 
periodic sets and sets defined by convergent sieves satisfying certain arithmetic equidistribution properties 
investigated by Br\"udern \cite{Bru2009} and subsequent authors\footnote{We note that, as is made clear in the 
cited papers here, the work of Br\"udern preceded and inspired the work of Schlage-Puchta and Keil, despite the 
apparent publication chronology.} (see \cite{Kei2011, Puc2002}). Given a set of natural numbers $\mathscr A$ 
having positive density and satisfying appropriate distribution properties in arithmetic progressions, Br\"udern 
makes use of the orthogonality relation
\[
\int_0^1\biggl| \sum_{n\in \mathscr A(N)}e(n\alpha)\biggr|^2\d\alpha =A(N),
\]
together with a major arc estimation
\[
\int_{\mathfrak M}\biggl| \sum_{n\in \mathscr A(N)}e(n\alpha)\biggr|^2\d\alpha \sim A(N),
\]
to deliver a minor arc estimate
\[
\int_\grm \biggl| \sum_{n\in \mathscr A(N)}e(n\alpha)\biggr|^2\d\alpha =o(N).
\]
Such a bound would follow from the estimate \eqref{1.3} for strongly subconvex $L^p$-sets $\mathscr A$ 
whenever one has in addition the non-trivial Weyl-type bound
\begin{equation}\label{1.10}
\sup_{\alpha\in \grm}\biggl| \sum_{n\in \mathscr A(N)}e(n\alpha)\biggr|=o(A(N)).
\end{equation}
The bound \eqref{1.3} is not accessible to Br\"udern's almost-periodic theory without highly stringent additional 
conditions on the set $\mathscr A$. Likewise, the estimate \eqref{1.10} does not seem to be easily accessible from 
the upper bound \eqref{1.3} without additional hypotheses. Thus, while there are certainly connections between the 
concepts of subconvex $L^p$-sets and Br\"udern's almost periodic sets, the associated theories seem for now to be 
largely disjoint. It would be very interesting to determine if strongly subconvex $L^p$-sets are characterised in 
terms of limit-periodic structures.\par 

Throughout, the letter $\varepsilon$ will denote a positive number. We adopt the convention that whenever 
$\varepsilon$ appears in a statement, either implicitly or explicitly, we assert that the statement holds for each 
$\varepsilon>0$. Our basic parameter will be $N$, a sufficiently large positive number. In addition, we use $\ll $ and 
$\gg$ to denote Vinogradov's well-known notation, implicit constants depending at most on $\varepsilon$, as well 
as other ambient parameters apparent from the context. We write $f\asymp g$ when $f\ll g$ and $g\ll f$. Also, we 
define $\|\theta\|$ and $\{ \theta\}$ for $\theta\in \mathbb R$ by putting 
$\|\theta\|=\min \{ |\theta -t|:t\in \mathbb Z\}$  and $\{ \theta\}=\theta -\lfloor \theta\rfloor$.

\section{Examples of subconvex $L^p$-sets}
We present a number of naturally occurring examples of subconvex $L^p$-sets in this section in order to motivate 
our subsequent discussions. Our first result concerns the classical and well-known example of the natural numbers 
$\mathbb N$. Here, we justify some of the remarks made in the introduction following Definition \ref{definition1.1}. 
The proof of this theorem follows as a familiar exercise. We provide details for the sake of completeness.

\begin{theorem}\label{theorem2.1}
One has
\[
\int_0^1 \biggl| \sum_{1\le n\le N}e(n\alpha)\biggr|\d\alpha \asymp \log (2N),
\]
and when $p>1$, one has
\[
\int_0^1\biggl| \sum_{1\le n\le N}e(n\alpha)\biggr|^p\d\alpha \ll_p N^{p-1}, 
\]
Thus, the set $\mathbb N$ is a weakly subconvex $L^1$-set, and when $p>1$ it is a strongly subconvex $L^p$-set.
\end{theorem}

\begin{proof} Summing over the implicit 
geometric progression, we find that when $\alpha\notin \mathbb Z$, then  for integral values of $N$ one has
\begin{equation}\label{2.1}
\biggl| \sum_{1\le n\le N}e(n\alpha)\biggr| =\Bigl| \frac{\sin(N\pi\alpha)}{\sin(\pi \alpha)}\Bigr| \ll \|\alpha\|^{-1}.
\end{equation}
Thus
\[
\int_0^1\biggl| \sum_{1\le n\le N}e(n\alpha)\biggr|\d\alpha \ll 1+\int_{1/N}^{1/2}\alpha^{-1}\d\alpha \ll \log (2N),
\]
and when $p>1$, meanwhile,
\[
\int_0^1\biggl| \sum_{1\le n\le N}e(n\alpha)\biggr|^p\d\alpha \ll N^{p-1}+\int_{1/N}^{1/2}\alpha^{-p}\d\alpha \ll 
N^{p-1}.
\]
The lower bound implicit in the first estimate of the theorem follows via a straightforward exercise using the 
formula \eqref{2.1}, and so the conclusion of the theorem is complete.
\end{proof}

Our second example is based on the relatively recent result of Keil \cite[Theorem 1.2]{Kei2013} mentioned in 
connection with the estimates \eqref{1.4} and \eqref{1.5}.

\begin{example}\label{example2.2} The set $\mathscr N_r$ of $r$-free numbers is a weakly subconvex 
$L^{1+1/r}$-set, and when $p>1+1/r$ it is a strongly subconvex $L^p$-set.
\end{example}

The final entry on our list of basic examples of subconvex $L^p$-sets makes use of certain Beatty sequences. Given 
real numbers $\alpha$ and $\beta$, we define the {\it Beatty set}
\[
\mathscr B(\alpha,\beta)=\{ n\in \mathbb N:\text{$n=\lfloor \alpha m+\beta\rfloor$ for some $m\in \mathbb N$}\}.
\]
Writing $\mathscr B(\alpha ,\beta)=\{n_1,n_2,\ldots \}$ with $1\le n_1<n_2<\ldots $, this set may be interpreted as a 
sequence $(n_k)_{k=1}^\infty$. In order to introduce the Beatty sequences of interest to us, we recall that a real 
number $\theta$ is said to be of finite Diophantine type if there is a natural number $k$ for which
\[
\liminf_{q\rightarrow \infty}q^k\|q\theta\|>0.
\]
Thus, the set of real numbers not of finite Diophantine type has measure $0$. Moreover, it follows from Liouville's 
theorem (see \cite[Theorem 1.1]{Bak2022}) that every real irrational algebraic number is of finite Diophantine type.

\par Our estimates for moments of exponential sums over Beatty sequences are essentially as strong as those given 
in Theorem \ref{theorem2.1} for the corresponding exponential sums over all natural numbers in an interval. 

\begin{theorem}\label{theorem2.3}
Suppose that $\alpha$ and $\beta$ are real numbers with $\alpha>0$ having the property that $1/\alpha$ is of 
finite Diophantine type. Then one has
\[
\int_0^1 \biggl| \sum_{\substack{1\le n\le N\\ n\in \mathscr B(\alpha,\beta)}}e(n\theta)\biggr|\d\theta \ll 
\log^2(2N),
\]
and when $p>1$, one has
\[
\int_0^1 \biggl| \sum_{\substack{1\le n\le N\\ n\in \mathscr B(\alpha,\beta)}}e(n\theta)\biggr|^p\d\theta \ll N^{p-1}.
\]
Thus, the set $\mathscr B(\alpha,\beta)$ is a weakly subconvex $L^1$-set, and when $p>1$ it is a strongly 
subconvex $L^p$-set.
\end{theorem}

\begin{proof} There is no loss of generality in restricting attention to the scenario in which $\alpha>1$. For if 
$0<\alpha \le 1$, then the set $\mathscr B(\alpha,\beta)$ contains all large natural numbers, and the desired 
conclusion is essentially immediate from Theorem \ref{theorem2.1}. Next, with our subsequent deliberations in 
mind, we equip ourselves with a conclusion concerning Diophantine approximations to the real number $1/\alpha$. 
Since $1/\alpha$ has finite Diophantine type, we may suppose that there is a natural number $k$ with the property 
that
\[
\liminf_{q\rightarrow \infty}q^k\| q/\alpha \|>0.
\]
In particular, there exists a positive number $c=c(\alpha)$ having the property that, for all large enough natural 
numbers $q$, and for all integers $a$, one has
\begin{equation}\label{2.2}
\Bigl| \frac{1}{\alpha}-\frac{a}{q}\Bigr|\ge \frac{c}{q^{k+1}}.
\end{equation}

\par Our plan is to adopt an argument relating exponential sums over Beatty sequences to corresponding 
unrestricted exponential sums, following a path close to that pursued in the proof of \cite[Lemma 4.3]{BGV2014}. 
Along the way, we incorporate adjustments and minor corrections relative to the latter source. We begin by 
observing that $n=\lfloor \alpha m+\beta\rfloor$ for some $m\in \mathbb Z$ if and only if
\begin{equation}\label{2.3}
1-\alpha^{-1}<\{ \alpha^{-1}(n-\beta)\}<1.
\end{equation}
For if $\alpha^{-1}(n-\beta)=m-1+\delta$, for some $\delta\in [0,1]$ and $m\in \mathbb Z$, then
\[
n=\alpha m+\beta -\alpha (1-\delta),
\]
whence
\[
\lfloor \alpha m+\beta\rfloor =\lfloor n+\alpha (1-\delta)\rfloor .
\]
This last integer is equal to $n$ if and only if $0\le \alpha (1-\delta)<1$. Thus, we have $\lfloor \alpha m+\beta
\rfloor =n$ if and only if $1-1/\alpha <\delta \le 1$, or equivalently, one has
\[
1-1/\alpha <\{ m-1+\delta \}\le 1.
\]
The situation with $\delta=1-1/\alpha$ corresponds to $n=\alpha m+\beta-1$, whilst that with $\delta=1$ 
corresponds to $n=\alpha m+\beta$. Since $\alpha $ is irrational, these situations can happen in their respective 
situations at most once.\par

Next, define the function $\psi:\mathbb R\rightarrow [-1/2,1/2)$ by putting
\[
\psi(x)=x-\lfloor x\rfloor -1/2,
\]
and observe that $\psi(x)$ is periodic with period $1$. When $x\in [0,1)$, moreover, one has
\[
\alpha^{-1}+\psi(x)-\psi(x+1/\alpha)=\begin{cases}1,&\text{when $1-1/\alpha\le x<1$},\\
0,&\text{when $0\le x<1-1/\alpha$}.\end{cases}
\]
Here, we have corrected the corresponding statement in the proof of \cite[Lemma 4.3]{BGV2014}. It therefore 
follows from the criterion \eqref{2.3} that
\begin{equation}\label{2.4}
\sum_{\substack{1\le n\le N\\ n\in \mathscr B(\alpha, \beta)}}e(n\theta)=S_1(\theta)+S_2(\theta),
\end{equation}
where
\[
S_1(\theta)=\frac{1}{\alpha}\sum_{1\le n\le N}e(n\theta)
\]
and
\[
S_2(\theta)=\sum_{1\le n\le N}\Bigl( \psi \Bigl( \frac{n-\beta}{\alpha}\Bigr) -\psi\Bigl( \frac{n+1-\beta}{\alpha}\Bigr) 
\Bigr) e(\theta n)+O(1).
\]

\par The most difficult part of the decomposition \eqref{2.4} concerns the expression $S_2(\theta)$. Again following 
\cite{BGV2014}, we write
\[
W(\varphi)=\sum_{1\le n\le N}\min\{ 1,N^{-k}\|n/\alpha -\varphi\|^{-1}\}.
\]
Then, by applying Montgomery and Vaughan \cite[Lemma D.1]{MV2007}, we see that
\begin{align}
S_2(\theta)=\,\sum_{0<|h|\le N^k}&\frac{e(-\beta h/\alpha)-e((1-\beta )h/\alpha)}{2\pi{\rm i}h}\sum_{1\le n\le N}
e(n(\theta +h/\alpha))\notag \\
&\, +O(1+W((\beta-1)/\alpha)+W(\beta/\alpha)).\label{2.5}
\end{align}
By Dirichlet's approximation theorem, there exist $a\in \mathbb Z$ and $q\in \mathbb N$ with $(a,q)=1$ and 
$1\le q\le N^k$ satisfying the bound $|q/\alpha -a|\le N^{-k}$. Since $\alpha$ has the property \eqref{2.2}, we find 
that
\[
cq^{-k-1}\le \Bigl| \frac{1}{\alpha}-\frac{a}{q}\Bigr|\le q^{-1}N^{-k},
\]
whence $q^k\ge cN^k$, and in particular $q\gg N$. We therefore deduce from Baker \cite[Lemma 3.2]{Bak1986} that
\begin{align*}
W(\varphi)&=N^{-k}\sum_{1\le n\le N}\min\{ N^k,\|n/\alpha -\varphi\|^{-1}\}\\
&\ll N^{-k}(N^k+q\log q)(N/q+1)\\
&\ll N\left( q^{-1}+N^{-1}+(q\log q)N^{-k-1}+(\log q)N^{-k}\right) .
\end{align*}
Since we may suppose that $N\ll q\le N^k$, we therefore see that $W(\varphi)\ll \log (2N)$, and hence we deduce 
from \eqref{2.5} that
\begin{equation}\label{2.6}
S_2(\theta)\ll \log (2N)+\sum_{0<|h|\le N^k}\frac{1}{|h|}\min \{ N, \|\theta +h/\alpha\|^{-1}\} .
\end{equation}

\par By a change of variable, we see that
\begin{align*}
\int_0^1|S_2(\theta)|\d\theta &\ll \log (2N)+\sum_{0<|h|\le N^k}\frac{1}{|h|}\int_0^1\min\{N,\|\theta\|^{-1}\} 
\d\theta \\
&\ll \log (2N)+\sum_{0<|h|\le N^k}\frac{1}{|h|}\log (2N),
\end{align*}
whence
\[
\int_0^1|S_2(\theta)|\d\theta \ll \log^2(2N).
\]
We therefore deduce from \eqref{2.4} and the argument associated with Theorem \ref{theorem2.1} that
\begin{align*}
\int_0^1\biggl| \sum_{\substack{1\le n\le N\\ n\in \mathscr B(\alpha,\beta)}}e(n\theta)\biggr|\d\theta 
&\le \frac{1}{\alpha}\int_0^1\biggl| \sum_{1\le n\le N}e(n\theta)\biggr|\d\theta +\int_0^1|S_2(\theta)|\d\theta \\
&\ll \frac{1}{\alpha}\log (2N)+\log^2(2N).
\end{align*}
The first conclusion of the theorem is now immediate.\par

The second bound of the theorem requires additional effort. We suppose throughout that $p$ is a real number with 
$p>1$. We observe first that from \eqref{2.6}, there is a choice of $\eta\in \{ +1,-1\}$ for which
\[
S_2(\theta)\ll\log (2N)+\sum_{1\le h\le N^k}\frac{1}{h}\min\{ N,\|h/\alpha +\eta \theta\|^{-1}\}.
\]
By dividing the summation over $h$ into dyadic intervals, we find that
\begin{equation}\label{2.7}
S_2(\theta)\ll \log (2N)+\sum_{j=0}^J 2^{-j}U(2^j),
\end{equation}
where $J=\lfloor k(\log N)/(\log 2)\rfloor$, and
\begin{equation}\label{2.8}
U(H)=\sum_{H\le h\le 2H}\min\{ N,\|h/\alpha +\eta \theta\|^{-1}\} .
\end{equation}

\par Consider a typical value of $H=2^j$ with $0\le j\le J$. By Dirichlet's approximation theorem, there exist 
$a\in \mathbb Z$ and $q\in \mathbb N$ with $(a,q)=1$ and $1\le q\le H$ satisfying the bound
\[
|q/\alpha -a|\le H^{-1}.
\]
Since $\alpha $ has the property \eqref{2.2}, we find that
\[
cq^{-k-1}\le \Bigl| \frac{1}{\alpha}-\frac{a}{q}\Bigr|\le q^{-1}H^{-1},
\]
whence $q\gg H^{1/k}$. We therefore deduce from Baker \cite[Lemma 3.2]{Bak1986} that
\begin{align*}
\sum_{H\le h\le 2H}\min\{ N,\|h/\alpha +\eta \theta\|^{-1}\}&\ll (N+q\log q)(H/q+1)\\
&\ll NH\Bigl( \frac{1}{q}+\frac{1}{H}+\frac{q\log q}{NH}+\frac{\log q}{N}\Bigr) .
\end{align*}
Since we may suppose that $H^{1/k}\ll q\le H\ll N^k$, we discern that
\begin{equation}\label{2.9}
U(H)=\sum_{H\le h\le 2H}\min\{ N,\| h/\alpha +\eta \theta\|^{-1}\} \ll NH^{1-1/k}\log (2H)\ll NH^{1-1/(2k)}.
\end{equation}

\par By applying H\"older's inequality and substituting the estimate \eqref{2.9}, we find from \eqref{2.8} that
\begin{align*}
U(H)^p&\ll U(H)^{(p-1)/2}H^{(p-1)/2}\sum_{H\le h\le 2H}\left( \min\{ N,\|h/\alpha +\eta \theta\|^{-1}\}
\right)^{(p+1)/2}\\
&\ll (NH^{1-1/(2k)})^{(p-1)/2}H^{(p-1)/2}\sum_{H\le h\le 2H}\left( \min\{ N,\|h/\alpha +\eta \theta\|^{-1}\}
\right)^{(p+1)/2}.
\end{align*}
Since we are assuming throughout that $p>1$, it follows via a change of variable that
\begin{align}
\int_0^1 U(H)^p\d\theta &\ll N^{(p-1)/2}H^{p-1-(p-1)/(4k)}\sum_{H\le h\le 2H}\int_0^1 
\left( \min\{ N,\|\beta\|^{-1}\}
\right)^{(p+1)/2}\d\beta \notag \\
&\ll N^{(p-1)/2}H^{p-(p-1)/(4k)}N^{(p+1)/2-1}\notag \\
&\ll N^{p-1}H^{p-2\delta (p-1)},\label{2.10}
\end{align}
where we write $\delta=1/(8k)$.\par

We now return to \eqref{2.7}, observing that an additional application of H\"older's inequality reveals that
\[
|S_2(\theta)|^p\ll \log^p(2N)+\biggl( \sum_{j=0}^J2^{-\delta j}\biggr)^{p-1}\sum_{j=0}^J 2^{(\delta (p-1)-p)j}
U(2^j)^p.
\]
By applying \eqref{2.10}, we therefore infer that
\begin{align*}
\int_0^1 |S_2(\theta )|^p\d\theta &\ll \log^p(2N)+\sum_{j=0}^J 2^{(\delta (p-1)-p)j}\int_0^1 U(2^j)^p\d\theta \\
&\ll \log^p(2N)+\sum_{j=0}^J 2^{(\delta(p-1)-p)j}N^{p-1}(2^j)^{p-2\delta (p-1)}\\
&\ll \log^p(2N)+N^{p-1}\sum_{j=0}^J2^{-\delta (p-1)j}\ll N^{p-1}.
\end{align*}
Finally, on substituting this bound into a consequence of the decomposition \eqref{2.4}, we arrive at the upper 
bound
\begin{align*}
\int_0^1 \biggl| \sum_{\substack{1\le n\le N\\ n\in \mathscr B(\alpha,\beta)}}e(n\theta )\biggr|^p\d\theta &\ll 
\int_0^1 |S_1(\theta)|^p\d\theta +\int_0^1|S_2(\theta)|^p\d\theta \\
&\ll \int_0^1 \biggl| \sum_{1\le n\le N}e(n\theta )\biggr|^p\d\theta +N^{p-1}.
\end{align*}
The second conclusion of the theorem is now immediate from the second estimate of Theorem \ref{theorem2.1}. 
This completes the proof of the theorem.
\end{proof}

It is possible that a more sophisticated argument might yield the bound
\[
\int_0^1\biggl|\sum_{\substack{1\le n\le N\\ n\in \mathscr B(\alpha,\beta)}}e(n\theta)\biggr|\d\theta \ll \log (2N).
\]
The reader should regard this as a challenge for the enthusiast.

\section{Engineering new subconvex $L^p$-sets}
Given one or more subconvex $L^p$-sets, it is possible to modify these sets in various ways so as to create new 
examples of subconvex $L^q$-sets, for values of $q$ depending on the original parameter $p$. In this section we 
survey the methods available for engineering these new subconvex sets. This survey of methods is far from 
exhaustive, but at least serves to illustrate that there is available an extraordinary diversity of examples of subconvex 
$L^p$-sets. Throughout this section, we concentrate on conclusions for {\it strongly} subconvex $L^p$-sets. 
However, all of our results apply equally well to {\it weakly} subconvex $L^p$-sets, with the corresponding proofs 
proceeding mutatis mutandis.\par

We begin by examining complements of sets.

\begin{theorem}\label{theorem3.1}
Suppose that $\mathscr A$ and $\mathscr B$ are strongly subconvex $L^p$-sets with $p>1$ and 
$\mathscr A\subseteq \mathscr B$, and define $\mathscr C=\mathscr B\setminus \mathscr A$. Then provided that 
$\mathscr C$ has positive lower density, the set $\mathscr C$ is a strongly subconvex $L^p$-set.
\end{theorem} 

\begin{proof} One has
\[
\sum_{n\in \mathscr C(N)}e(n\alpha )=\sum_{n\in \mathscr B(N)}e(n\alpha) -\sum_{n\in \mathscr A(N)}e(n\alpha).
\]
It therefore follows from \eqref{1.1} and \eqref{1.3} that
\[
I_p(N;\mathscr C)\ll I_p(N;\mathscr B)+I_p(N;\mathscr A)\ll N^{p-1}.
\]
Hence, provided that $\mathscr C$ has positive lower density, we see that it is a strongly subconvex $L^p$-set.
\end{proof}

\begin{corollary}\label{corollary3.2} Suppose that $\mathscr A$ is a strongly subconvex $L^p$-set with $p>1$ for 
which $\mathscr A^c=\mathbb N\setminus \mathscr A$ has positive lower density. Then $\mathscr A^c$ is a 
strongly subconvex $L^p$-set.
\end{corollary}

\begin{proof} The set $\mathbb N$ is a strongly subconvex $L^q$-set whenever $q>1$, and hence also a strongly 
subconvex $L^p$-set. Thus, the conclusion is immediate from Theorem \ref{theorem3.1}.
\end{proof}

\begin{corollary}\label{corollary3.3} When $r>s\ge 2$, the set of $r$-free numbers that are not $s$-free, namely 
$\mathscr N_r\setminus \mathscr N_s$, is a strongly subconvex $L^p$-set whenever $p>1+1/s$.
\end{corollary}

\begin{proof} As noted in Example \ref{example2.2}, both $\mathscr N_r$ and $\mathscr N_s$ are strongly 
subconvex $L^p$-sets whenever $p>1+1/s$. Since $\mathscr N_s\subset \mathscr N_r$ and 
$\mathscr N_r\setminus \mathscr N_s$ has positive density, the desired conclusion follows from Theorem 
\ref{theorem3.1}.
\end{proof}

\par Next, we turn our attention to translations and dilations of sets.

\begin{theorem}\label{theorem3.4}
Suppose that $\mathscr A$ is a strongly subconvex $L^p$-set with $p>1$. Then, whenever $q\in \mathbb N$ and 
$a\in \mathbb Z$, the set $\mathscr B=(q\mathscr A+a)\cap \mathbb N$ is also a strongly subconvex $L^p$-set.
\end{theorem}

\begin{proof} We have
\[
\sum_{n\in \mathscr B(N)}e(n\alpha)=\sum_{n\in \mathscr A(N/q)}e((qn+a)\alpha)+O(1),
\]
whence, by a change of variable, it follows from \eqref{1.1} that
\begin{align*}
I_p(N;\mathscr B)&\ll 1+\int_0^1\biggl| \sum_{n\in \mathscr A(N/q)}e(qn\alpha)\biggr|^p\d\alpha \\
&\ll 1+I_p(N/q;\mathscr A)\ll N^{p-1}.
\end{align*}
The desired conclusion follows at once.
\end{proof}

Perturbations of subconvex $L^p$-sets, in which sufficiently few elements are removed or added, also yield 
subconvex $L^p$-sets.

\begin{theorem}\label{theorem3.5}
Suppose that $\mathscr A$ is a strongly subconvex $L^p$-set with $p>1$. Let $\mathscr B$ and $\mathscr C$ be 
subsets of $\mathbb N$ with $\mathscr B\subset \mathscr A$ and 
$\mathscr C\subset \mathbb N\setminus \mathscr A$ satisfying the property that
\begin{equation}\label{3.1}
B(N)+C(N)\ll N^{2-2/p}.
\end{equation}
Then the perturbed set $\mathscr A'=(\mathscr A\setminus \mathscr B)\cup \mathscr C$ is also a strongly 
subconvex $L^p$-set.
\end{theorem}

\begin{proof} One has
\[
\sum_{n\in \mathscr A'(N)}e(n\alpha)=\sum_{n\in \mathscr A(N)}e(n\alpha)-\sum_{n\in \mathscr B(N)}e(n\alpha)+
\sum_{n\in \mathscr C(N)}e(n\alpha),
\]
whence, it follows from \eqref{1.1} that
\[
I_p(N;\mathscr A')\ll I_p(N;\mathscr A)+I_p(N;\mathscr B)+I_p(N;\mathscr C).
\]
Since we may assume that $p<2$, it follows from H\"older's inequality and orthogonality that one has
\[
I_p(N;\mathscr B)\le I_2(N;\mathscr B)^{p/2}\le B(N)^{p/2},
\]
with a similar upper bound available for the mean value over the exponential sum associated with the set 
$\mathscr C(N)$. Since $\mathscr A$ is a strongly subconvex $L^p$-set, it therefore follows from the hypothesis 
\eqref{3.1} that
\begin{align*}
I_p(N;\mathscr A')&\ll N^{p-1}+B(N)^{p/2}+C(N)^{p/2}\\
&\ll N^{p-1}+(N^{2-2/p})^{p/2}\ll N^{p-1}.
\end{align*}
This confirms that $\mathscr A'$ is a strongly subconvex $L^p$-set.  
\end{proof}

A particularly attractive feature of subconvex $L^p$-sets is that intersections of such sets, provided that these 
intersections remain of positive lower density, are likewise subconvex $L^q$-sets, for an appropriate choice of $q$. 
The condition here that the intersections have positive lower density is natural. For example, the set $\mathscr N_2$ 
consisting of square-free numbers plainly has the property that $\mathscr N_2\cap \mathscr N_2^c$ is empty. So 
although $\mathscr N_2$ and $\mathscr N_2^c$ are both weakly subconvex $L^{3/2}$-sets, their intersection is 
not. The problem of determining whether or not two subconvex $L^p$-sets $\mathscr A$ and $\mathscr B$ have 
an intersection with positive lower density may have a number-theoretic flavour. In particular, in many examples, 
the circle method will provide a viable means of establishing whether or not $\mathscr A\cap \mathscr B$ has 
positive lower density.\par

It is convenient in the proof of the next theorem, and elsewhere, to define an exponential sum 
$f(\alpha;\mathscr D)=f_N(\alpha;\mathscr D)$ associated with each subset $\mathscr D$ of the natural numbers. 
Thus, we write
\begin{equation}\label{3.2}
f_N(\alpha;\mathscr D)=\sum_{n\in \mathscr D(N)}e(n\alpha).
\end{equation}

\begin{theorem}\label{theorem3.6}
Let $\mathscr A$ be a strongly subconvex $L^q$-set with $q>1$, and let $\mathscr B$ be a strongly subconvex 
$L^r$-set with $r>1$. Define the real number $p$ via the relation
\begin{equation}\label{3.3}
\frac{1}{p}=\frac{1}{q}+\frac{1}{r}-1.
\end{equation}
Suppose that $\mathscr A\cap \mathscr B$ has positive lower density, and in addition one has
\begin{equation}\label{3.4}
\frac{3}{2}<\frac{1}{q}+\frac{1}{r}<2.
\end{equation}
Then the set $\mathscr A\cap \mathscr B$ is a strongly subconvex $L^p$-set.
\end{theorem}

\begin{proof} The desired conclusion is essentially a consequence of Young's convolution inequality (see 
\cite{BL1976}, for example), though the argument is sufficiently simple and instructive that we take the opportunity 
to describe our proof in detail. We recall that as a consequence of Definition \ref{definition1.1}, it is implicit that the 
strong subconvexity of $\mathscr A$ and $\mathscr B$ implies that $q<2$ and $r<2$. Define the real number $p$ 
via the relation \eqref{3.3}, and note that the condition \eqref{3.4} then ensures that $1<p<2$. Also, put 
$\mathscr C=\mathscr A\cap \mathscr B$. Then on making use of the notation \eqref{3.2}, it follows from 
orthogonality that one has
\[
f(\alpha;\mathscr C)=\int_0^1 f(\alpha-\beta;\mathscr A)f(\beta ;\mathscr B)\d\beta .
\]

\par Next, observe that as a consequence of the last relation, it follows from \eqref{1.1} that
\begin{align*}
I_p(N;\mathscr C)&=\int_0^1 |f(\alpha;\mathscr C)|^{p-1}\Bigl| \int_0^1 f(\alpha-\beta;\mathscr A)
f(\beta;\mathscr B)\d\beta \Bigr| \d\alpha \\
&\le \int_0^1 \int_0^1 |f(\alpha;\mathscr C)|^{p-1}|f(\alpha-\beta;\mathscr A)f(\beta;\mathscr B)|\d\beta \d\alpha .
\end{align*}
An application of H\"older's inequality therefore reveals that
\begin{equation}\label{3.5}
I_p(N;\mathscr C)\le T_1^{\frac{1}{q}-\frac{1}{p}}T_2^{\frac{1}{r}-\frac{1}{p}}T_3^{\frac{1}{p}},
\end{equation}
where
\begin{align*}
T_1&=\int_0^1\int_0^1 |f(\alpha;\mathscr C)|^p|f(\alpha-\beta;\mathscr A)|^q\d\beta \d\alpha ,\\
T_2&=\int_0^1\int_0^1 |f(\alpha;\mathscr C)|^p|f(\beta;\mathscr B)|^r\d\beta \d\alpha ,\\
T_3&=\int_0^1\int_0^1 |f(\alpha-\beta;\mathscr A)|^q|f(\beta;\mathscr B)|^r\d\beta \d\alpha .
\end{align*}
Here, we have made use of the observation that
\[
\frac{1}{q}-\frac{1}{p}=1-\frac{1}{r},
\]
so that $0<1/q-1/p<1$, and similarly $0<1/r-1/p<1$.  Moroever,
\[
\Bigl( \frac{1}{q}-\frac{1}{p}\Bigr) +\Bigl( \frac{1}{r}-\frac{1}{p}\Bigr) +\frac{1}{p}=\frac{1}{q}+\frac{1}{r}-\frac{1}{p}=1.
\]
Hence, our application of H\"older's inequality is legitimate.\par

By applying a change of variables, and utilising the definition \eqref{1.1} and the fact that $\mathscr A$ is a strongly 
subconvex $L^q$-set, we find that
\[
T_1=I_p(N;\mathscr C)I_q(N;\mathscr A)\ll N^{q-1}I_p(N;\mathscr C).
\]
Similarly, one sees that $T_2\ll N^{r-1}I_p(N;\mathscr C)$, and
\[
T_3=I_q(N;\mathscr A)I_r(N;\mathscr B)\ll N^{q+r-2}.
\]
We therefore deduce from \eqref{3.5} that
\begin{align}
I_p(N;\mathscr C)&\ll \left( N^{q-1}I_p(N;\mathscr C)\right)^{\frac{1}{q}-\frac{1}{p}} 
\left( N^{r-1}I_p(N;\mathscr C)\right)^{\frac{1}{r}-\frac{1}{p}} \left( N^{q+r-2}\right)^{1/p}\notag \\
&\ll N^{1-1/p}I_p(N;\mathscr C)^{1-1/p}.\label{3.6}
\end{align}

\par Our hypothesis that the set $\mathscr C=\mathscr A\cap \mathscr B$ has positive lower density ensures that 
when $0\le \alpha \le (100N)^{-1}$, one has
\[
|f(\alpha;\mathscr C)|=\biggl| \sum_{n\in \mathscr C(N)}e(n\alpha)\biggr|\ge \frac{1}{2}|C(N)|\gg N,
\]
whence
\[
I_p(N;\mathscr C)\ge \int_0^{1/(100N)}|f(\alpha ;\mathscr C)|^p\d\alpha \gg N^{p-1}.
\]
Then we conclude from \eqref{3.6} that $I_p(N;\mathscr C)^{1/p}\ll N^{1-1/p}$, whence 
$I_p(N;\mathscr C)\ll N^{p-1}$. This confirms that $\mathscr A\cap \mathscr B$ is a strongly subconvex $L^p$-set.
\end{proof}

This conclusion yields an immediate corollary for intersections of subconvex $L^p$-sets with Beatty sequences.

\begin{corollary}\label{corollary3.7}
Let $\mathscr A$ be a strongly subconvex $L^q$-set for some real number $q\in (1,2)$. Suppose that $\alpha$ 
and $\beta$ are real numbers with $\alpha>0$ such that $1/\alpha$ is of finite Diophantine type. Then the set 
$\mathscr A\cap \mathscr B(\alpha,\beta)$ is a strongly subconvex $L^p$-set whenever $p$ is a real number with 
$p\in (q,2)$ and $\mathscr A\cap \mathscr B(\alpha,\beta)$ has positive lower density.
\end{corollary}

\begin{proof} The Beatty sequence $\mathscr B(\alpha,\beta)$ is a strongly subconvex $L^r$-set whenever 
$1<r<2$, as a consequence of Theorem \ref{theorem2.3}. Choose any such value of $r$ sufficiently close to $1$, 
and define the real number $p$ by means of the relation \eqref{3.3}. Then we may take $p$ as close as we like to 
$q$ by choosing $r$ suitably close to $1$. The conclusion of Theorem \ref{theorem3.6} then shows that 
$\mathscr A\cap \mathscr B(\alpha,\beta)$ is a strongly subconvex $L^p$-set provided that this set has positive 
density and condition \eqref{3.4} is satisfied. Since $q>1$ and $r>1$, the upper bound on $1/q+1/r$ in the latter 
condition is automatically satisfied. Moreover, one has $q<2$, so the lower bound $1/q+1/r>3/2$ is satisfied on 
taking $r$ sufficiently close to $1$. The desired conclusion therefore follows.
\end{proof}

The discussion of this section exhibits numerous means of generating new examples of subconvex $L^p$-sets from 
any existing examples. One can also splice together different sets by considering (in the obvious sense) such hyprid 
examples as
\[
\mathscr D(N)=\mathscr A(N/3)\cup (\mathscr B(2N/3)\setminus \mathscr B(N/3))\cup 
(\mathscr C(N)\setminus \mathscr C(2N/3)),
\]
when given subconvex $L^p$-sets $\mathscr A$, $\mathscr B$ and $\mathscr C$. A basic problem worthy of 
attention would be to characterise subconvex $L^p$-sets.

\begin{problem}\label{problem3.8} Suppose that $p$ is a real number with $1<p<2$. Characterise, to the extent 
possible, strongly subconvex $L^p$-sets.
\end{problem}

\section{Weyl sums over weakly subconvex $L^p$-sets, I}
We begin with the most striking of the applications of subconvex $L^p$-sets to exponential sums over polynomials 
having summands restricted to these $L^p$-sets, demonstrating how to establish the upper bound \eqref{1.9} of 
Theorem \ref{theorem1.2}.

\begin{proof}[Proof of Theorem \ref{theorem1.2}]
 Let $\mathscr A$ be a weakly subconvex $L^p$-set for some real number $p$ with $1\le p\le 4/3$. We may 
suppose that $k\ge 3$, and that $a\in \mathbb Z$ and $q\in \mathbb N$ satisfy $(a,q)=1$ and 
$|\alpha_k-a/q|\le q^{-2}$. We recall the definition \eqref{1.6}, and introduce the exponential sum 
$F_k(\boldsymbol \alpha)=F_k(\boldsymbol \alpha;N)$, defined by 
\begin{equation}\label{4.1}
F_k(\boldsymbol \alpha;N)=\sum_{n\in \mathscr A(N)}e(\psi(n;\boldsymbol \alpha)).
\end{equation}
We also make use of the exponential sum $g(\alpha)=f_N(\alpha;\mathscr A)$ defined via \eqref{3.2}, so that
\begin{equation}\label{4.2}
g(\alpha)=\sum_{n\in \mathcal A(N)}e(n\alpha),
\end{equation}
and introduce the exponential sum
\begin{equation}\label{4.3}
G_k(\boldsymbol \alpha ,\beta)=\sum_{1\le n\le N}e(\psi(n;\boldsymbol \alpha )+\beta n).
\end{equation}
We note that $G_k(\boldsymbol \alpha,\beta)$ is merely a more convenient form of an exponential sum that may 
be written in terms of the sum $\Psi_k(\boldsymbol \alpha;N)$ defined in \eqref{1.7}. It now follows from 
orthogonality that one has the fundamental relation
\begin{equation}\label{4.4}
F_k(\boldsymbol \alpha)=\int_0^1G_k(\boldsymbol \alpha,\beta)g(-\beta)\d\beta .
\end{equation}

\par An application of H\"older's inequality reveals that
\begin{equation}\label{4.5}
|F_k(\boldsymbol \alpha)|\le \Bigl( \int_0^1|G_k(\boldsymbol \alpha,\beta)|^4\d\beta \Bigr)^{1/4}
\Bigl( \int_0^1|g(\beta )|^{4/3}\d\beta \Bigr)^{3/4}.
\end{equation}
Here, our hypothesis that $\mathscr A$ is a weakly subconvex $L^p$-set for some real number $p$ with $p\le 4/3$ 
ensures that
\[
\int_0^1|g(\beta)|^{4/3}\d\beta \le N^{4/3-p}\int_0^1|g(\beta )|^p\d\beta \ll N^{1/3+\varepsilon}.
\]
Hence, we deduce from \eqref{4.5} that
\begin{equation}\label{4.6}
|F_k(\boldsymbol \alpha)|^4\ll N^{1+\varepsilon}\int_0^1|G_k(\boldsymbol \alpha,\beta)|^4\d\beta .
\end{equation}

\par By orthogonality, the integral on the right hand side of \eqref{4.6} may be interpreted in the form
\[
\int_0^1|G_k(\boldsymbol \alpha ,\beta)|^4\d\beta =\sum_{\substack{1\le n_1,n_2,n_3,n_4\le N\\ n_1+n_2=n_3+n_4
}}e(\psi(n_1;\boldsymbol \alpha)+\psi(n_2;\boldsymbol \alpha)-
\psi(n_3;\boldsymbol \alpha)-\psi(n_4;\boldsymbol \alpha)).
\]
For suitable integers $h_1$ and $h_2$ with $|h_i|<N$ $(i=1,2)$, one can write the summands $n_3$ and $n_4$ in
the shape
\[
n_3=n_1+h_1\quad \text{and}\quad n_4=n_1+h_2.
\]
The relation $n_1+n_2=n_3+n_4$ then implies that $n_2=n_1+h_1+h_2$. Thus, making use of the forward 
difference operator
\[
\Del_1(\varphi(x);h)=\varphi(x+h)-\varphi(x),
\]
and the second order operator
\[
\Del_2(\varphi(x);h_1,h_2)=\Del_1(\Del_1(\varphi(x);h_1);h_2),
\]
we see that
\[
\psi(n_1;\boldsymbol \alpha)+\psi(n_2;\boldsymbol \alpha)-\psi(n_3;\boldsymbol \alpha)-
\psi(n_4;\boldsymbol \alpha)=\Del_2(\psi(n_1;\boldsymbol \alpha);h_1,h_2).
\]
In this way, we conclude that
\begin{equation}\label{4.7}
\int_0^1|G_k(\boldsymbol \alpha ,\beta)|^4\d\beta =\sum_{|h_1|<N}\sum_{|h_2|<N}\sum_{n\in \mathscr I(N;\bfh)}
e(\Delta_2(\psi(n;\boldsymbol \alpha );h_1,h_2)),
\end{equation}
in which $\mathscr I(N;\bfh)$ denotes the interval of integers defined by
\[
\mathscr I(N;\bfh)=[1,N]\cap [1-h_1,N-h_1]\cap [1-h_2,N-h_2]\cap [1-h_1-h_2,N-h_1-h_2].
\]
The right hand side of \eqref{4.7} will be recognised as the principal output of a second order Weyl differencing 
process applied to the exponential sum $G_k(\boldsymbol \alpha ,0)=\Psi_k(\boldsymbol \alpha;N)$. The interested reader 
should direct their attention to \cite[Lemma 2.3]{Vau1997} and compare with the situation therein with $j=2$.\par

On substituting \eqref{4.7} into \eqref{4.6}, we see that
\[
|F_k(\boldsymbol \alpha)|^{2^2}\ll N^{1+\varepsilon}\sum_{|h_1|<N}\sum_{|h_2|<N}T_2,
\]
where
\[
T_2=\sum_{x\in \mathscr I_2}e(\Delta_2(\psi(n;\boldsymbol \alpha);h_1,h_2)),
\]
in which $\mathscr I_2=\mathscr I(N;\bfh)$. A comparison of this upper bound with the statement of 
\cite[Lemma 2.3]{Vau1997} reveals, just as in the proof of \cite[Lemma 2.4]{Vau1997}, that the application of $k-3$ 
further Weyl differencing steps delivers the bound
\[
|F_k(\boldsymbol \alpha)|^{2^{k-1}}\ll N^{2^{k-1}-k+\varepsilon}\sum_{|h_1|<N}\cdots \sum_{|h_{k-1}|<N}
\sum_{x\in \mathscr I_{k-1}}e(h_1\cdots h_{k-1}p_{k-1}(x;h_1,\ldots ,h_{k-1})),
\]
where $\mathscr I_{k-1}$ is a subinterval of integers lying within $[1,N]$, and
\[
p_{k-1}(x;h_1,\ldots ,h_{k-1})=k!\alpha_k(x+\tfrac{1}{2}h_1+\ldots+\tfrac{1}{2}h_{k-1})+(k-1)!\alpha_{k-1}.
\]
Hence, just as in the completion of the argument of the proof of \cite[Lemma 2.4]{Vau1997}, one concludes that
\[
|F_k(\boldsymbol \alpha)|^{2^{k-1}}\ll N^{2^{k-1}+\varepsilon}(q^{-1}+N^{-1}+qN^{-k}).
\]
Thus
\[
|F_k(\boldsymbol \alpha )|\ll N^{1+\varepsilon}(q^{-1}+N^{-1}+qN^{-k})^{2^{1-k}},
\]
confirming the desired conclusion \eqref{1.9}.
\end{proof}

We next address the proof of Theorem \ref{theorem1.3}. This we achieve in two stages, the second of which we 
defer to the next section. We recall the definition of the exponent $\sigma_p(k)$ recorded in the preamble to the 
statement of Theorem \ref{theorem1.3}.

\begin{lemma}\label{lemma4.1}
Suppose that $\mathscr A$ is a weakly subconvex $L^p$-set for some real number $p$ with $1\le p<2$, and 
further that $k\ge 2$. Let $(\alpha_0,\alpha_1,\ldots ,\alpha_k)\in \mathbb R^{k+1}$, and suppose that 
$a\in \mathbb Z$ and $q\in \mathbb N$ satisfy $(a,q)=1$ and $|\alpha_k-a/q|\le q^{-2}$. Then, for each 
$\varepsilon>0$ and each large real number $N$, one has
\begin{equation}\label{4.8}
\sum_{n\in \mathscr A(N)}e(\alpha_kn^k+\ldots +\alpha_1n+\alpha_0)\ll 
N^{1+\varepsilon}(q^{-1}+N^{-1}+qN^{-k})^{\sigma_p(k)}.
\end{equation}
\end{lemma}

\begin{proof} We begin by noting that when $\mathscr A$ is a weakly subconvex $L^1$-set, then it is also a weakly 
subconvex $L^p$-set for any $p>1$. The conclusion of Theorem \ref{theorem1.3} for $p=1$ is obtained from that 
for $1<p<2$ by taking $p$ sufficiently close to $1$, on adjusting the value of $\varepsilon$ in the desired 
conclusion. There is consequently no loss of generality in supposing throughout that $1<p<2$ and $k\ge 2$. Next, 
we observe that in the case $k=2$ it makes little sense to perform two Weyl differencing steps, so that the proof of 
Theorem \ref{theorem1.2} is of limited use to us. In this case we nonetheless have the orthogonality relation 
\eqref{4.4}, but in present circumstances we apply H\"older's inequality to obtain the upper bound
\begin{equation}\label{4.9}
|F_k(\boldsymbol \alpha)|\le \Bigl( \int_0^1|G_k(\boldsymbol \alpha, \beta)|^{p/(p-1)}\d\beta \Bigr)^{1-1/p}
\Bigl( \int_0^1|g(\beta)|^p\d\beta \Bigr)^{1/p}.
\end{equation}
Our hypothesis that $\mathcal A$ is a weakly subconvex $L^p$-set ensures via \eqref{1.2} that
\begin{equation}\label{4.10}
\int_0^1|g(\beta)|^p\d\beta \ll N^{p-1+\varepsilon}.
\end{equation}
Moreover, since $1<p<2$, we may suppose that $p/(p-1)>2$. Hence, by orthogonality, we see that
\begin{align}
\int_0^1|G_k(\boldsymbol \alpha,\beta)|^{\frac{p}{p-1}}\d\beta &\le 
\Bigl( \sup_{\beta\in [0,1)}|G_k(\boldsymbol \alpha,\beta)|\Bigr)^{\frac{2-p}{p-1}}
\int_0^1 |G_k(\boldsymbol \alpha,\beta)|^2\d\beta \notag \\
& \le N\Bigl( \sup_{\beta\in [0,1)}|G_k(\boldsymbol \alpha ,\beta)|\Bigr)^{\frac{2-p}{p-1}}.\label{4.11}
\end{align}

\par Since $a\in \mathbb Z$ and $q\in \mathbb N$ satisfy $(a,q)=1$ and $|\alpha_k-a/q|\le q^{-2}$, it follows from 
Weyl's inequality (see \cite[Lemma 2.4]{Vau1997}) that we have the bound
\[
G_k(\boldsymbol \alpha ,\beta)\ll N^{1+\varepsilon}(q^{-1}+N^{-1}+qN^{-k})^{2^{1-k}}.
\]
Thus we see from \eqref{4.11} that
\[
\int_0^1|G_k(\boldsymbol \alpha,\beta)|^{\frac{p}{p-1}}\d\beta \ll N\cdot N^{\frac{2-p}{p-1}+\varepsilon}
(q^{-1}+N^{-1}+qN^{-k})^{\left( \frac{2-p}{p-1}\right)2^{1-k}}.
\]
On substituting this estimate together with \eqref{4.10} into \eqref{4.9}, we conclude that
\begin{align*}
F_k(\boldsymbol \alpha)&\ll N^{\frac{1}{p}+\varepsilon}(q^{-1}+N^{-1}+qN^{-k})^{\left( \frac{2}{p}-1\right)2^{1-k}}
(N^{p-1+\varepsilon})^{\frac{1}{p}}\\
&\ll N^{1+2\varepsilon}(q^{-1}+N^{-1}+qN^{-k})^{\left( \frac{2}{p}-1\right) 2^{1-k}}.
\end{align*}
Thus far, our discussion has been independent of the value of $k\ge 2$. Specialising to the situation with $k=2$, 
however, we obtain the estimate \eqref{4.8} with $\sigma_p(k)=\frac{1}{p}-\frac{1}{2}$, confirming the conclusion of 
the lemma in this case.\par

The situation with $k\ge 3$ and $1\le p\le 4/3$ is immediate from Theorem \ref{theorem1.2}, since, in such 
circumstances, one has $\sigma_p(k)=2^{1-k}$. We turn now to address the situation in which $k\ge 3$ and 
$4/3<p<2$. Here, our starting point is the relation \eqref{4.4}, though we apply H\"older's inequality in a manner 
that incorporates the $4$-th moment of $G_k(\boldsymbol \alpha,\beta)$. Thus, we obtain
\begin{equation}\label{4.12}
|F_k(\boldsymbol \alpha)|\le U_1^{\frac{3}{2}-\frac{2}{p}}U_2^{\frac{1}{p}-\frac{1}{2}}\Bigl( 
\int_0^1 |g(\beta)|^p\d\beta \Bigr)^{1/p},
\end{equation}
where
\begin{equation}\label{4.13}
U_r=\int_0^1|G_k(\boldsymbol \alpha ,\beta)|^{2r}\d\beta \quad (r\ge 1).
\end{equation}
By orthogonality, one finds from \eqref{4.3} that $U_1\le N$. Meanwhile, the argument leading from \eqref{4.7} to 
the conclusion of the proof of Theorem \ref{theorem1.2} shows that
\[
U_2\ll N^{3+\varepsilon}(q^{-1}+N^{-1}+qN^{-k})^{2^{3-k}}.
\]
Thus, on recalling the upper bound \eqref{4.10}, we deduce from \eqref{4.12} that
\begin{align*}
F_k(\boldsymbol \alpha)&\ll N^{\frac{3}{2}-\frac{2}{p}+\varepsilon}\left( N^3(q^{-1}+N^{-1}+qN^{-k})^{2^{3-k}}
\right)^{\frac{1}{p}-\frac{1}{2}}(N^{p-1})^{\frac{1}{p}}\\
&\ll N^{1+\varepsilon}(q^{-1}+N^{-1}+qN^{-k})^{\left( \frac{1}{p}-\frac{1}{2}\right)2^{3-k}}.
\end{align*}
Since $\sigma_p(k)=\bigl( \frac{1}{p}-\frac{1}{2}\bigr)2^{3-k}$ when $k\ge 3$ and $4/3<p<2$, this confirms the 
conclusion of the lemma in this final case. 
\end{proof}

\section{Weyl sums over weakly subconvex $L^p$-sets, II}
We next shift attention to the application of Vinogradov's methods for bounding exponential sums, though now 
with the variables restricted to subconvex $L^p$-sets. Here, we adopt an approach that differs from that more 
familiarly presented in work, for example, of Vaughan \cite[\S5.2]{Vau1997}.

\begin{lemma}\label{lemma5.1}
Suppose that $\mathscr A$ is a weakly subconvex $L^p$-set for some real number $p$ with $1\le p<2$, and 
further that $k\ge 3$. Let $(\alpha_0,\alpha_1,\ldots ,\alpha_k)\in \mathbb R^{k+1}$, and suppose that 
$a\in \mathbb Z$ and $q\in \mathbb N$ satisfy $(a,q)=1$ and $|\alpha_k-a/q|\le q^{-2}$. Then, for each 
$\varepsilon>0$ and each large real number $N$, one has
\begin{equation}\label{5.1}
\sum_{n\in \mathscr A(N)}e(\alpha_kn^k+\ldots +\alpha_1n+\alpha_0)\ll 
N^{1+\varepsilon}(q^{-1}+N^{-1}+qN^{-k})^{\tau_p(k)}.
\end{equation}
\end{lemma}

\begin{proof} Just as in the proof of Lemma \ref{lemma4.1}, there is no loss of generality in supposing throughout 
that $1<p<2$. Furthermore, in view of the definition of $\tau_p(k)$, we may suppose in addition that $k\ge 3$. 
Moreover, we again have the upper bounds \eqref{4.9} and \eqref{4.10} employed in the course of the proof of 
Lemma \ref{lemma4.1}. We augment these estimates with a new mean value estimate. Write
\begin{equation}\label{5.2}
r=\tfrac{1}{2}k(k-1).
\end{equation}
It transpires that we are able to bound $|F_k(\boldsymbol \alpha)|$ in terms of the mean value $U_r$, defined as in 
\eqref{4.13}. It is this mean value to which we apply Vinogradov's methods.\par

Observe first that from \eqref{4.3}, one has
\[
G_k(\boldsymbol \alpha ,\beta)=
\frac{1}{N}\sum_{1\le n\le N}\sum_{1-n\le h\le N-n}e(\psi(n+h;\boldsymbol \alpha)+(n+h)\beta).
\]
Thus, by H\"older's inequality, one sees that
\begin{equation}\label{5.3}
|G_k(\boldsymbol \alpha,\beta)|^{2r}\le \frac{1}{N}\sum_{1\le n\le N}|E_k(\boldsymbol \alpha,\beta;n)|^{2r},
\end{equation}
in which we write
\[
E_k(\boldsymbol \alpha,\beta;n)=\sum_{1-n\le h\le N-n}e(\psi(n+h;\boldsymbol \alpha)+h\beta).
\]
Next, write
\[
\varphi(h;\boldsymbol \beta)=\beta_{k-2}h^{k-2}+\ldots +\beta_1h,
\]
and put
\[
D_k(\boldsymbol \alpha ,\boldsymbol \beta;n)=\sum_{1-n\le h\le N-n}
e(\psi(n+h;\boldsymbol \alpha)+\varphi(h;\boldsymbol \beta)).
\]
Then we see via orthogonality that
\[
\int_0^1|E_k(\boldsymbol \alpha ,\beta;n)|^{2r}\d\beta =\sum_{|m_2|\le 2rN^2}\cdots \sum_{|m_{k-2}|\le 2rN^{k-2}}
\int_{[0,1)^{k-2}}|D_k(\boldsymbol \alpha ,\boldsymbol \beta;n)|^{2r}
e(-\boldsymbol \beta\cdot \bfm)\d\boldsymbol \beta ,
\]
where we abbreviate $\beta_2m_2+\ldots +\beta_{k-2}m_{k-2}$ to $\boldsymbol \beta\cdot \bfm$. An application 
of the triangle inequality therefore propels us from \eqref{5.3} to the bound
\begin{equation}\label{5.4}
\int_0^1|G_k(\boldsymbol \alpha,\beta)|^{2r}\d\beta \ll \frac{1}{N}\sum_{1\le n\le N}N^{\frac{1}{2}(k-1)(k-2)-1}
\int_{[0,1)^{k-2}}|D_k(\boldsymbol \alpha ,\boldsymbol \beta;n)|^{2r}\d\boldsymbol \beta .
\end{equation}

\par By orthogonality, the mean value on the right hand side of \eqref{5.4} is equal to
\[
\sum_{1-n\le h_1\le N-n}\cdots \sum_{1-n\le h_{2r}\le N-n}e(\sigma(n,\bfh;\boldsymbol \alpha)),
\]
in which we write
\[
\sigma(n,\bfh;\boldsymbol \alpha)=\sum_{i=1}^r \left( \psi(n+h_i;\boldsymbol \alpha)-\psi(n+h_{r+i};\boldsymbol 
\alpha)\right) ,
\]
and the summation over $h_1,\ldots ,h_{2r}$ is subject to the condition
\[
\sum_{i=1}^r(h_i^j-h_{r+i}^j)=0\quad (1\le j\le k-2).
\]
By expanding the polynomial $\psi(n+h;\boldsymbol \alpha)$ using the binomial theorem, moreover, we see that 
for each summand $\bfh$, we have
\[
\sigma (n,\bfh;\boldsymbol \alpha)=\alpha_k\sum_{i=1}^r(h_i^k-h_{r+i}^k)+(kn\alpha_k+\alpha_{k-1})\sum_{i=1}^r
(h_i^{k-1}-h_{r+i}^{k-1}).
\]
Hence, with the same implicit condition on $h_1,\ldots ,h_{2r}$, we deduce that
\begin{align}
\biggl| \sum_{1\le n\le N}\sum_{1-n\le h_1\le N-n}&\cdots \sum_{1-n\le h_{2r}\le N-n}
e(\sigma(n,\bfh;\boldsymbol \alpha))\biggr| \notag \\
&\le \sum_{|h_1|\le N}\cdots \sum_{|h_{2r}|\le N}\biggl| \sum_{n\in \mathscr I(\bfh)}
e\biggl( kn\alpha_k\sum_{i=1}^r(h_i^{k-1}-h_{r+i}^{k-1})\biggr) \biggl|,\label{5.5}
\end{align}
where $\mathscr I(\bfh)$ denotes the set of integers $n$ in the interval $[1,N]$ satisfying the property that
\[
1-h_i\le n\le N-h_i\quad (1\le i\le 2r).
\]

\par Since the set $\mathscr I(\bfh)$ is a subinterval of the integers lying in $[1,N]$, it follows in a standard manner 
that by summing the implicit geometric progression, one has
\[
\sum_{n\in \mathscr I(\bfh)}e\biggl( kn\alpha_k\sum_{i=1}^r(h_i^{k-1}-h_{r+i}^{k-1})\biggr) \ll \min\biggl\{
N,\biggl\| k\alpha_k\sum_{i=1}^r (h_i^{k-1}-h_{r+i}^{k-1})\biggr\|^{-1}\biggr\} .
\]
By substituting this bound into \eqref{5.5}, and thence into \eqref{5.4}, we conclude thus far that
\begin{equation}\label{5.6}
\int_0^1|G_k(\boldsymbol \alpha ,\beta)|^{2r}\d\beta \ll N^{\frac{1}{2}(k-1)(k-2)-2}\sum_{|m_{k-1}|\le 2rN^{k-1}}
\rho(m_{k-1})\min\{ N,\|km\alpha_k\|^{-1}\},
\end{equation}
where we write $\rho(m)$ for the number of solutions of the system of equations
\begin{align*}
\sum_{i=1}^r(h_i^{k-1}-h_{r+i}^{k-1})&=m,\\
\sum_{i=1}^r(h_i^j-h_{r+i}^j)&=0\quad (1\le j\le k-2),
\end{align*}
with $|h_i|\le N$ $(1\le i\le 2r)$.\par

A standard application of the recent resolution of the main conjecture on Vinogradov's mean value theorem (see 
\cite{BDG2016}, \cite{Woo2016}, \cite{Woo2019}) shows via orthogonality that
\begin{align*}
\rho(m)&=\int_{[0,1)^{k-1}}\biggl| \sum_{|h|\le N}e(\gamma_1h+\ldots +\gamma_{k-1}h^{k-1})\biggr|^{2r}
e(-\gamma_{k-1}m)\d\boldsymbol \gamma \\
&\ll 1+\int_{[0,1)^{k-1}}\biggl| \sum_{1\le h\le N}e(\gamma_1h+\ldots +\gamma_{k-1}h^{k-1})\biggr|^{2r}
\d\boldsymbol \gamma \\
&\ll N^{r+\varepsilon}+N^{2r-\frac{1}{2}k(k-1)}.
\end{align*}
Thus, we infer from \eqref{5.2} and  \eqref{5.6} that
\[
\int_0^1 |G_k(\boldsymbol \alpha ,\beta)|^{2r}\d\beta \ll N^{2r-k-1+\varepsilon}\sum_{|m_{k-1}|\le 2rN^{k-1}}
\min\{ N,\|km\alpha_k\|^{-1}\} .
\]
A standard reciprocal sums lemma (see, for example \cite[Lemma 2.2]{Vau1997}) leads from here to the estimate
\begin{equation}\label{5.7}
\int_0^1 |G_k(\boldsymbol \alpha,\beta)|^{2r}\d\beta \ll N^{2r-1+\varepsilon}(q^{-1}+N^{-1}+qN^{-k}).
\end{equation}

\par We may now return to combine \eqref{5.7} with the bounds \eqref{4.9} and \eqref{4.10} already established. 
Observe first that when $p/(p-1)=2r$, which is to say that $p=2r/(2r-1)$, the estimate \eqref{5.7} may be substituted 
into \eqref{4.9} along with \eqref{4.10} to give
\begin{align*}
F_k(\boldsymbol \alpha)&\ll N^\varepsilon \left( N^{2r-1}(q^{-1}+N^{-1}+qN^{-k})\right)^{1/(2r)}(N^{p-1})^{1/p}\\
&\ll N^{1+\varepsilon}(q^{-1}+N^{-1}+qN^{-k})^{1/(2r)}.
\end{align*}
Since any weakly subconvex $L^q$-set with $q<2r/(2r-1)$ is also a weakly subconvex $L^p$-set with $p=2r/(2r-1)$, 
and
\[
\frac{2r}{2r-1}=1+\frac{1}{k^2-k-1},
\]
this establishes \eqref{5.1} when $k\ge 3$ and $1\le p\le 1+1/(k^2-k-1)$.\par

The situation with $1+1/(k^2-k-1)<p<2$ requires more care. Applying H\"older's inequality to \eqref{4.4}, we 
obtain the upper bound
\begin{equation}\label{5.8}
|F_k(\boldsymbol \alpha)|\le U_1^{\omega_1}U_r^{\omega_r}\biggl( \int_0^1|g(\beta )|^p\d\beta \biggr)^{1/p},
\end{equation} 
where $U_1$ and $U_r$ are defined as in \eqref{4.13}, and
\[
\omega_r=\Bigl( \frac{2}{p}-1\Bigr)\frac{1}{k^2-k-2}\quad \text{and}\quad \omega_1=1-\frac{1}{p}-\omega_r.
\]
Notice here that since
\[
\frac{k^2-k}{k^2-k-1}<p<2,
\]
one has
\[
0<\frac{2}{p}-1<\frac{k^2-k-2}{k^2-k}\quad \text{and}\quad \frac{1}{k^2-k}<1-\frac{1}{p}<1.
\]
It follows that $\omega_r<1-1/p$, whence our application of H\"older's inequality was indeed legitimate.\par

The bound $U_1\le N$ again follows via orthogonality, while the bound \eqref{5.7} supplies an estimate for $U_r$. 
Consequently, on substituting these estimates along with \eqref{4.10} into \eqref{5.8}, we deduce that
\begin{align*}
F_k(\boldsymbol \alpha)&\ll N^\varepsilon (N)^{\omega_1}(N^{2r-1}(q^{-1}+N^{-1}+qN^{-k}))^{\omega_r}
(N^{p-1})^{1/p}\\
&\ll N^{1+\varepsilon}(q^{-1}+N^{-1}+qN^{-k})^{\omega_r}.
\end{align*}
Since $\tau_p(k)=\omega_r$ in present circumstances, this completes the proof of the lemma.
\end{proof}

The proof of Theorem \ref{theorem1.3} is completed by combining the conclusions of Lemmata \ref{lemma4.1} and 
\ref{lemma5.1}. The former addresses the situations with $k\ge 2$ when $\sigma_p(k)\ge \tau_p(k)$, and the latter 
addresses those scenarios in which $k\ge 3$ and $\tau_p(k)\ge \sigma_p(k)$.\par

Enthusiastic readers might choose to entertain themselves by adapting these methods to show that when $k\ge 6$ 
and $1\le p\le 4/3$, one has an analogue of Heath-Brown's refinement of Weyl's inequality for weakly subconvex 
$L^p$-sets $\mathscr A$. Thus, when $\alpha\in \mathbb R$, and $a\in \mathbb Z$ and $q\in \mathbb N$ satisfy 
$(a,q)=1$ and $|\alpha -a/q|\le q^{-2}$, one may adapt the proof of \cite[Theorem 1]{HB1988} to show that
\[
\sum_{n\in \mathscr A(N)}e(\alpha n^k)\ll N^{1+\varepsilon}(N(q^{-1}+N^{-3}+qN^{-k}))^{\frac{4}{3}2^{-k}}.
\]

\section{Equidistribution and subconvex $L^p$-sets}
In order to explore the topic of equidistribution of polynomial sequences with arguments from subconvex 
$L^p$-sets, we apply Weyl's criterion. Recall that the sequence $(a_n)_{n=1}^\infty$ is equidistributed modulo $1$ 
if and only if for any integer $m\ne 0$, we have
\[
\lim_{N\rightarrow \infty}\frac{1}{N}\biggl| \sum_{n=1}^N e(ma_n)\biggr|=0.
\]
Define the polynomial $\psi(x;\boldsymbol \alpha)$ as in \eqref{1.6} with fixed coefficients 
$\boldsymbol \alpha\in \mathbb R^{k+1}$. Then as a particular consequence of work of Weyl \cite{Wey1916}, the 
sequence $(\psi(n;\boldsymbol\alpha))_{n=1}^\infty$ is equidistributed modulo $1$ whenever one at least of the 
coefficients $\alpha_2,\ldots ,\alpha_k$ is irrational. In these circumstances, we see that for each non-zero integer 
$m$, one has
\[
\biggl| \sum_{n=1}^Ne(m\psi(n;\boldsymbol \alpha))\biggr|=o(N).
\]
Our objective in this section is to establish an analogue of this last conclusion in which the summands $n$ are 
restricted to a subconvex $L^p$-set. We begin with the proof of Theorem \ref{theorem1.4}, which concerns strongly 
subconvex $L^p$-sets.

\begin{proof}[The proof of Theorem \ref{theorem1.4}]
Let $\mathscr A=\{a_1,a_2,\ldots\}$, with $a_1<a_2<\ldots $, be a strongly subconvex $L^p$-set with $1\le p<2$. In 
order to facilitate tidier notation, we make a modest preliminary manoeuvre. Since $\mathscr A$ has positive lower 
density in $\mathbb N$, we see that there exists a positive number $c=c(\mathscr A)$ having the property that 
when $M$ is large, there exists an integer $N$ with $N\le cM$ for which one has 
$\{a_1,a_2,\ldots ,a_M\}=\mathscr A(N)$. Next, let $k\ge 2$, and define $\psi(n;\boldsymbol \alpha)$ according to 
\eqref{1.6}. The hypotheses of the theorem permit us the assumption that one at least of the coefficients 
$\alpha_2,\ldots ,\alpha_k$ is irrational. We may therefore suppose that there is an integer $h$ with $2\le h\le k$ for 
which $\alpha_k,\ldots ,\alpha_{h+1}$ are all rational, and $\alpha_h$ is irrational. We may suppose in addition that 
$r$ is a natural number having the property that $r\alpha_j\in \mathbb Z$ $(h+1\le j\le k)$. In this context, if $h=k$, 
then we set $r=1$.\par

According to Weyl's criterion, the sequence $(\psi(a_n;\boldsymbol \alpha))_{n=1}^\infty$ is equidistributed 
modulo $1$ if and only if for any integer $m\ne 0$, one has
\[
\lim_{M\rightarrow \infty}\frac{1}{M}\biggl| \sum_{\nu=1}^M e(m\psi(a_\nu;\boldsymbol \alpha))\biggr|=0.
\]
The latter holds if and only if
\[
\lim_{N\rightarrow \infty}\frac{1}{N}\biggl| \sum_{n\in \mathscr A(N)}e(m\psi(n;\boldsymbol \alpha))\biggr|=0.
\]
Thus, in the notation introduced in \eqref{4.1}, our goal is to show that for any integer $m\ne 0$, one has
\begin{equation}\label{6.1}
|F_k(m\boldsymbol \alpha;N)|=o(N).
\end{equation}
Applying the notation \eqref{4.2} and \eqref{4.3} employed in the course of the proof of Theorem \ref{theorem1.2}, 
we may infer from \eqref{4.4} that
\[
F_k(m\boldsymbol \alpha;N)=\int_0^1 G_k(m\boldsymbol \alpha,\beta)g(-\beta)\d\beta .
\]
An application of H\"older's inequality reveals that
\[
|F_k(m\boldsymbol \alpha;N)|\le \Bigl( \sup_{\beta\in [0,1)}|G_k(m\boldsymbol \alpha,\beta)|\Bigr)^{\frac{2}{p}-1}
\Bigl( \int_0^1|G_k(m\boldsymbol \alpha,\beta)|^2\d\beta\Bigr)^{1-\frac{1}{p}}\Bigl( \int_0^1|g(\beta)|^p\d\beta 
\Bigr)^{\frac{1}{p}}.
\]
Since $\mathscr A$ is a strongly subconvex $L^p$-set with $1\le p<2$, we have
\[
\int_0^1|g(\beta)|^p\d\beta \ll N^{p-1}.
\]
Then it follows via orthogonality that
\begin{equation}\label{6.2}
F_k(m\boldsymbol \alpha;N)\ll (N)^{1-\frac{1}{p}}(N^{p-1})^{\frac{1}{p}}\Bigl( \sup_{\beta\in [0,1)}
|G_k(m\boldsymbol \alpha,\beta)|\Bigr)^{\frac{2}{p}-1}.
\end{equation}
Thus, we may infer that the estimate \eqref{6.1} holds whenever one has
\begin{equation}\label{6.3}
\sup_{\bet\in [0,1)}|G_k(m\boldsymbol \alpha,\beta)|=o(N).
\end{equation}

\par We establish \eqref{6.3} in a fairly routine manner via a sequence of intermediate steps. We begin by 
considering an auxiliary polynomial
\[
\omega(t;\boldsymbol \beta)=\beta_ht^h+\ldots +\beta_1t+\beta_0,
\]
and the associated exponential sum
\[
U(\boldsymbol \beta ;N)=\sum_{1\le n\le N}e(\omega (n;\boldsymbol \beta)).
\]
Suppose that $\beta_h$ is irrational. We claim that for any fixed value of $\eta>1$, there exists a natural number 
$N_0$ having the property that
\begin{equation}\label{6.4}
\sup_{\beta_0,\ldots ,\beta_{h-1}}|U(\boldsymbol \beta;N_0)|<N_0/\eta.
\end{equation}
We establish this claim by contradiction. Suppose, if possible, that for all large values of $N$, one has
\[
\sup_{\beta_0,\ldots ,\beta_{h-1}}|U(\boldsymbol \beta;N)|\ge N/\eta.
\]
Then it follows from an appropriate version of Weyl's inequality (see \cite[Theorem 5.1]{Bak1986}) that there exist 
$a\in \mathbb Z$ and $q\in \mathbb N$ with $(a,q)=1$,
\[
q<\eta^hN^{1/3}\quad \text{and}\quad |q\beta_h-a|<\eta^hN^{1/3-h}.
\]
For each large natural number $Q$, we fix $N=\lfloor (Q/\eta^h)^3\rfloor$. Then we find that one has $q<Q$ and 
$|q\beta_h-a|<Q^{-3/2}$. Since these inequalities hold for all large values of $Q$, it follows via a consideration of 
the continued fraction convergents to $\beta_h$ that this coefficient must be rational (compare 
\cite[Lemma 2.7]{LLW2025}). This deduction contradicts our assumption that $\beta_h$ is irrational, and so we are 
forced to conclude that our claim \eqref{6.4} is indeed valid.\par

Now we return to consider the exponential sum $G_k(m\boldsymbol \alpha,\beta)$ and the bound \eqref{6.3} that 
we seek to establish. Let $\eta>1$ be arbitrary. Put $\beta_h=mr^h\alpha_h$. Then since 
$r\alpha_j\in \mathbb Z$ $(h+1\le j\le k)$, we see that
\[
m\psi(x+ry;\boldsymbol \alpha)+\beta (x+ry)\equiv \beta_hy^h+\ldots +\beta_1y+\beta_0\mmod{1},
\]
for some real numbers $\beta_0,\ldots ,\beta_{h-1}$ depending at most on $r$, $m$, $x$, $\beta$ and 
$\boldsymbol \alpha$. Notice that since $\alpha_h$ is presumed to be irrational, then so too is $\beta_h$. Let 
$N_0$ be the natural number furnished in association with the upper bound \eqref{6.4} that we have established. 
When $N$ is large enough in terms of $r$ and $N_0$, one has
\[
\biggl| \sum_{1\le n\le N}e(m\psi(n;\boldsymbol \alpha)+\beta n)-\frac{1}{N_0}\sum_{1\le x\le N}
\sum_{1\le y\le N_0}e(m\psi(x+ry;\boldsymbol \alpha)+\beta (x+ry))\biggr| \le 2rN_0.
\]
Then we deduce from \eqref{6.4} that
\begin{align*}
\frac{1}{N}\sup_{\beta\in [0,1)}|G_k(m\boldsymbol \alpha,\beta)|&\le \frac{1}{N_0}\sup_{\beta_0,\ldots ,\beta_{h-1}}
|U(\boldsymbol \beta;N_0)|+\frac{2rN_0}{N}\\
&<\frac{1}{\eta}+\frac{2rN_0}{N}.
\end{align*}
Therefore, for each integer $m\ne 0$, we conclude that for every real number $\eta>1$, one has
\[
\limsup_{N\rightarrow \infty}\frac{1}{N}\sup_{\beta\in [0,1)}|G_k(m\boldsymbol \alpha,\beta)|\le \frac{1}{\eta},
\]
Since this upper bound holds for all $\eta>1$, we infer that
\[
\lim_{N\rightarrow \infty}\frac{1}{N}\sup_{\beta\in [0,1)}|G_k(m\boldsymbol \alpha,\beta)|=0,
\]
and the desired conclusion \eqref{6.3} follows. We therefore conclude that \eqref{6.1} holds, and so the sequence 
$(\psi(a_n;\boldsymbol \alpha))_{n=1}^\infty$ is equidistributed modulo $1$. This completes the proof of the 
theorem.
\end{proof}

The reader might well be puzzled that the conclusion of Theorem \ref{theorem1.4} applies only to polynomials of 
the shape $\psi(t;\boldsymbol \alpha)=\alpha_kt^k+\ldots +\alpha_1t+\alpha_0$ when one at least of 
$\alpha_2,\ldots , \alpha_k$ is irrational. It transpires that this condition is essential. Were it to be the case that 
$\alpha_2,\ldots ,\alpha_k$ are rational and $\alpha_1$ is irrational, then it is possible that the sequence 
$(\psi(a_n;\boldsymbol \alpha))_{n=1}^\infty$ is not equidistributed modulo $1$. Consider, for example, an 
irrational number $\theta\in (0,1)$ of finite Diophantine type and the set
\[
\mathscr A=\{ \lfloor n/\theta\rfloor : n\in \mathbb N\}.
\]
On the one hand, this Beatty sequence $\{a_1,a_2,\ldots \}$, with $a_1<a_2<\ldots $, has the property that 
$\{ a_n\theta\}\in (1-\theta,1)$ for each $n\in \mathbb N$. On the other hand, as a consequence of Theorem 
\ref{theorem2.3}, we see that $\mathscr A$ is a strongly subconvex $L^p$-set whenever $p>1$. It is therefore not 
possible that the sequence $(a_n\theta)_{n=1}^\infty$ is equidistributed modulo $1$. A modest adjustment of this 
argument would deliver the same conclusion for the polynomial $\alpha_kt^k+\ldots +\alpha_1t+\alpha_0$ when 
$\alpha_k,\ldots ,\alpha_2$ are all rational and $\alpha_1=\theta$.\par

We finish this section by noting that a conclusion similar to Theorem \ref{theorem1.4} may be established in which 
the sequence $\mathscr A$ is assumed only to be a {\it weakly} subconvex $L^p$-set with $1\le p<2$.

\begin{theorem}\label{theorem6.1}
Suppose that $\mathscr A=\{a_1,a_2,\ldots \}$, with $a_1<a_2<\ldots $, is a weakly subconvex $L^p$-set with 
$1\le p<2$. Let $k\ge 2$, suppose that $(\alpha_0,\alpha_1,\ldots ,\alpha_k)\in \mathbb R^{k+1}$, and define the 
polynomial $\psi(x;\boldsymbol \alpha)$ as in \eqref{1.6}. Suppose in addition that one at least of the coefficients 
$\alpha_2,\ldots,\alpha_k$ is of finite Diophantine type. Then the sequence 
$(\psi(a_n;\boldsymbol \alpha))_{n=1}^\infty$ is equidistributed modulo $1$.
\end{theorem}

\begin{proof}
We may proceed just as in the proof of Theorem \ref{theorem1.4}, though the hypothesis that $\mathscr A$ be 
weakly subconvex instead of strongly subconvex requires us to replace \eqref{6.2} by the bound
\begin{equation}\label{6.5}
F_k(m\boldsymbol \alpha)\ll (N)^{1-\frac{1}{p}}(N^{p-1+\varepsilon})^{\frac{1}{p}}
\Bigl( \sup_{\beta\in [0,1)}|G_k(m\boldsymbol \alpha ,\beta)|\Bigr)^{\frac{2}{p}-1}.
\end{equation}
Let $\varepsilon$ be a small positive number. Then it follows from Baker \cite[Theorem 5.1]{Bak1986} that whenever 
$P>N^{1-2^{-k}}$ and $|G_k(m\boldsymbol\alpha ,\beta)|>P$, then there exist $q\in \mathbb N$ and 
$a_1,\ldots ,a_k\in \mathbb Z$ having the property that
\[
q<N^\varepsilon (NP^{-1})^k,\quad (q,a_1,\ldots ,a_k)=1,\quad (q,a_2,\ldots ,a_k)\le N^\varepsilon,
\]
\[
|qm\alpha_j-a_j|<(NP^{-1})^kN^{\varepsilon-j}\quad (2\le j\le k)
\]
and
\[
|q(m\alpha_1+\beta)-a_1|<(NP^{-1})^kN^{\varepsilon-1}.
\]

\par We may suppose that there is an index $l$ with $2\le l\le k$ for which $\alpha_l$ is of finite Diophantine type. 
Thus, for some natural number $r$ and positive number $c$, it follows that for all $a\in \mathbb Z$ one has
\begin{equation}\label{6.6}
|qm\alpha_l-a|\ge c(qm)^{-r}.
\end{equation}
We take $\varepsilon=1/(k2^kr)$ and $P=N^{1-1/(k2^kr)}$. Then we deduce from the above discussion that 
whenever
\[
|G_k(m\boldsymbol \alpha ,\beta)|>N^{1-1/(k2^kr)},
\]
there exist $q\in \mathbb N$ and $a\in \mathbb Z$ with
\[
q<N^{2/(2^kr)}=(N^{1/r})^{2^{1-k}}
\]
and
\[
|qm\alpha_l-a|<(N^{1/r})^{2^{1-k}}N^{-l}.
\]
But then we find from \eqref{6.6} that
\[
c(qm)^{-r}\le |qm\alpha_l-a|<(N^{1/r})^{2^{1-k}}N^{-l}<N^{-3/2},
\]
whence for large enough values of $N$, one sees that $q>N^{1/r}$. Thus, we deduce that
\[
N^{1/r}<q<(N^{1/r})^{2^{1-k}},
\]
leading to a contradiction. We must therefore conclude that whenever $N$ is sufficiently large in terms of $m$ and 
$\alpha_l$, one has
\[
|G_k(m\boldsymbol \alpha,\beta)|\le N^{1-1/(k2^kr)},
\]
uniformly in $\beta$. We thus deduce from \eqref{6.5} that
\[
F_k(m\boldsymbol \alpha)\ll N^{1-\left(\frac{2}{p}-1\right) /(k2^kr)},
\]
whence, for each integer $m\ne 0$, one has
\[
\lim_{N\rightarrow \infty}N^{-1}|F_k(m\boldsymbol \alpha)|=0.
\]
This shows that
\[
\lim_{N\rightarrow \infty}\frac{1}{N}\biggl| \sum_{n=1}^Ne(m\psi(a_n;\boldsymbol \alpha))\biggr|=0,
\]
so that $(\psi(a_n;\boldsymbol \alpha))_{n=1}^\infty$ is equidistributed modulo $1$. This completes the proof of 
the theorem.
\end{proof}

\section{Subconvex $L^p$-sets and arithmetic functions}
We take the opportunity in this section to discuss the application of our earlier ideas to the problem of estimating 
averages of arithmetic functions with arguments restricted to subconvex $L^p$-sets $\mathscr A$. In many 
situations of practical interest, such as when $\mathscr A$ consists of square-free numbers, for example, one may 
apply direct approaches in which the arithmetic structure of $\mathscr A$ is employed to successfully analyse such 
averages. Our purpose in this section is two-fold. On the one hand, we show that the tools available for subconvex 
$L^p$-sets provided elegantly simple approaches to the estimation of such averages in which technical tedium is 
avoided. On the other hand, our approach is quite general, so that it applies to subconvex $L^p$-sets possessing 
no evident arithmetic structure.\par

In order to describe the most general form of our conclusions, we must introduce some fairly general notation. 
Consider an arithmetic function $f:\mathbb N\rightarrow \mathbb C$, and introduce the exponential sum 
$H_f(\alpha;N)$ defined by
\[
H_f(\alpha;N)=\frac{1}{N}\sum_{1\le n\le N}f(n)e(n\alpha).
\]
Also, with the large real number $N$ suppressed from the notation, write
\[
\| H_f\|_\infty =\sup_{\alpha\in [0,1)}|H_f(\alpha ;N)|
\]
and the normalised $l^2$-norm
\[
\|f\|_{l^2(N)}=\biggl( \frac{1}{N}\sum_{1\le n\le N}|f(n)|^2\biggr)^{1/2}.
\]

\begin{theorem}\label{theorem7.1}
Suppose that $\mathscr A$ is a strongly subconvex $L^p$-set for some real number $p$ with $1\le p<2$. Let 
$f:\mathbb N\rightarrow \mathbb C$ be any arithmetic function. Then one has
\[
\frac{1}{N}\sum_{n\in \mathscr A(N)}f(n)\ll \| H_f\|_\infty^{\frac{2}{p}-1}\|f\|_{l^2(N)}^{2-\frac{2}{p}}.
\]
\end{theorem}

\begin{proof} Recalling our earlier notation \eqref{4.2}, we see that
\[
\frac{1}{N}\sum_{n\in \mathscr A(N)}f(n)=\int_0^1 H_f(\alpha;N)g(-\alpha)\d\alpha .
\]
Hence, by H\"older's inequality and Parseval's identity, we find that
\begin{align*}
\frac{1}{N}\biggl| \sum_{n\in \mathscr A(N)}f(n)\biggr| &\le \Bigl( \sup_{\alpha\in [0,1)}|H_f(\alpha;N)|
\Bigr)^{\frac{2}{p}-1}\Bigl( \int_0^1 |H_f(\alpha;N)|^2\d\alpha \Bigr)^{1-\frac{1}{p}}
\Bigl( \int_0^1|g(\alpha)|^p\d\alpha \Bigr)^{1/p}\\
&\ll \|H_f\|_\infty^{\frac{2}{p}-1}\left( N^{-1}\|f\|_{l^2(N)}^2\right)^{1-\frac{1}{p}}(N^{p-1})^{\frac{1}{p}},
\end{align*}
and the desired conclusion follows.
\end{proof}

We now derive some corollaries of this theorem in order to illustrate the ease with which consequences may be 
extracted within this framework. We first consider the situation in which the arithmetic function 
$f(n)$ is given by the M\"obius function $\mu(n)$.

\begin{proof}[The proof of Theorem \ref{theorem1.5}] It is a consequence of work of Davenport 
\cite[equation (1)]{Dav1937} that for any positive number $A$, one has
\[
\frac{1}{N}\sum_{1\le n\le N}\mu(n)e(n\theta)\ll \log^{-A}(2N),
\]
uniformly in $\theta$. Thus we have
\[
\|H_\mu\|_{\infty}\ll \log^{-A}(2N).
\]
By a trivial estimate, moreover, one has
\[
\|\mu\|_{l^2(N)}=\Bigl( \frac{1}{N}\sum_{1\le n\le N}\mu(n)^2\Bigr)^{1/2}\le 1.
\]
Hence, we deduce from Theorem \ref{theorem7.1} that whenever $\mathscr A$ is a strongly subconvex $L^p$-set 
for some $1\le p<2$, then
\[
\frac{1}{N}\sum_{n\in \mathscr A(N)}\mu(n)\ll \left( \log^{-A}(2N)\right)^{\frac{2}{p}-1},
\]
whence for any $B>0$, one has
\[
\sum_{n\in \mathscr A(N)}\mu(n)\ll_B N\log^{-B}(2N).
\]
This completes the proof of Theorem \ref{theorem1.5}.
\end{proof}

One can modify the summands here in a pedestrian manner. Thus, for example, one can make use of the ideas of 
Davenport \cite{Dav1937}, along the lines pursued by Hua \cite[Chapters 5 and 6]{Hua1965}, to show that for each 
$A>0$ the bound
\[
\sum_{1\le n\le N}\mu(n)e(n^k\theta+n\beta )\ll N\log^{-A}(2N)
\]
holds uniformly in $\theta$ and $\beta$. Putting $f(n)=\mu(n)e(n^k\theta)$, one then has
\[
\|H_f\|_\infty\ll \log^{-A}(2N)\quad \text{and}\quad \|f\|_{l^2(N)}\le 1,
\]
and so the argument above, mutatis mutandis, yields the following corollary.

\begin{corollary}\label{corollary7.2} Suppose that $\mathscr A$ is a strongly subconvex $L^p$-set for some real 
number $p$ with $1\le p<2$. Let $k\in \mathbb N$ and $\alpha\in \mathbb R$. Then, whenever $A>0$ and $N$ is 
sufficiently large in terms of $A$, one has
\[
\sum_{n\in \mathscr A(N)}\mu(n)e(n^k\alpha)\ll N(\log N)^{-A}.
\]
\end{corollary}

A modest adjustment of this example yields a conclusion useful in applications of the circle method to additive 
problems involving primes restricted to subconvex $L^p$-sets. Here, as usual, we write $\Lambda (n)$ for the 
von Mangoldt function.

\begin{corollary}\label{corollary7.3} For each $k\in \mathbb N$, there is a positive number $\sigma=\sigma(k)$ 
with the following property. Suppose that $\mathscr A$ is a strongly subconvex $L^p$-set for some real 
number $p$ with $1\le p<2$. Let $A$ be a large positive real number, and suppose that $N$ is sufficiently large in 
terms of $A$. Suppose also that $\theta$ is a real number having the property that, whenever $a\in \mathbb Z$ 
and $q\in \mathbb N$ satisfy $(a,q)=1$ and $|q\theta -a|\le \log^A(2N)N^{-1}$, then $q>\log^A(2N)$. Then one 
has
\[
\sum_{n\in \mathscr A(N)}\Lambda(n)e(n^k\theta)\ll N\log^{-B}(2N),
\]
where $B=\bigl( \frac{2}{p}-1\bigr) \sigma A-1+\frac{1}{p}$.
\end{corollary}

\begin{proof} The hypotheses concerning $\theta$ ensure that the methods of Vinogradov and Hua (see 
\cite{Hua1965, Vin1947}) show that there is a positive number $\sigma=\sigma(k)$ having the property that
\[
\sum_{1\le n\le N}\Lambda (n)e(n^k\theta +\beta n)\ll N\log^{-\sigma A}(2N),
\]
uniformly in $\beta$. Thus, with $f(n)=\Lambda(n)e(n^k\theta)$, one has 
$\|H_f\|_\infty \ll \log^{-\sigma A}(2N)$
and
\[
\|f\|_{l^2(N)}\le \biggl( \frac{1}{N}\sum_{1\le n\le N}\Lambda(n)^2\biggr)^{1/2}\ll (\log (2N))^{1/2}.
\]
Hence, we deduce from Theorem \ref{theorem7.1} that
\[
\sum_{n\in \mathscr A(N)}\Lambda(n)e(n^k\theta)\ll N\left( \log^{-\sigma A}(2N)\right)^{\frac{2}{p}-1}
(\log (2N))^{1-\frac{1}{p}}.
\]
The desired conclusion now follows.
\end{proof}

By way of illustrating the potential for applications, we note that on taking $A$ sufficiently large, the conclusion of 
Corollary \ref{corollary7.3} would offer minor arc estimates of use in investigating $k$-th powers of prime numbers 
$\pi$ in which $\pi+2$ belongs to a subconvex $L^p$-set such as a Beatty sequence, or the set of squarefree 
numbers, or indeed the squarefree numbers belonging to a Beatty sequence.\par

Estimates for exponential sums with cusp form coefficients exhibit non-trivial cancellation when restricted to 
subconvex $L^p$-sets.

\begin{corollary}\label{corollary7.4}
Suppose that
\[
\sum_{n=1}^\infty a(n)e(nz)
\]
is a cusp form of weight $k>0$ for a discrete group $\Gamma$ for which $\infty$ is a cusp of width $1$. Then, 
whenever $\mathscr A$ is a strongly subconvex $L^p$-set for some real number $p$ with $1\le p<2$, one has
\[
\sum_{n\in \mathscr A(N)}a(n)\ll N^{\frac{k}{2}+1-\frac{1}{p}}(\log (2N))^{\frac{2}{p}-1}.
\]
\end{corollary}

\begin{proof} We apply Theorem \ref{theorem7.1} with $f(n)=a(n)$. Here, as a consequence of Iwaniec 
\cite[Theorem 5.3]{Iwa1991}, one has
\[
\|H_f\|_\infty =\sup_{\beta\in [0,1)}\biggl| \frac{1}{N}\sum_{1\le n\le N}a(n)e(\beta n)\biggr| \ll N^{-1+k/2}\log (2N),
\]
whilst \cite[Theorem 5.1]{Iwa1991} provides the bound
\[
\|f\|_{l^2(N)}=\biggl( \frac{1}{N}\sum_{1\le n\le N}|a(n)|^2\biggr)^{1/2}\ll N^{(k-1)/2}.
\]
Consequently, we see from Theorem \ref{theorem7.1} that
\begin{align*}
\sum_{n\in \mathscr A(N)}a(n)&\ll N\left( N^{\frac{k}{2}-1}\log (2N)\right)^{\frac{2}{p}-1}
\left( N^{\frac{k-1}{2}}\right)^{2-\frac{2}{p}}\\
&\ll N^{\frac{k}{2}+1-\frac{1}{p}}(\log (2N))^{\frac{2}{p}-1}.
\end{align*}
 This completes the proof of the corollary.
\end{proof}

By way of comparison, the bound \cite[Theorem 5.3]{Iwa1991} shows that
\begin{equation}\label{7.1}
\sum_{1\le n\le N}a(n)\ll N^{k/2}\log (2N),
\end{equation}
whilst by applying the trivial estimate $|a(n)|\ll n^{k/2}$, one finds that
\[
\biggl|\sum_{n\in \mathscr A(N)}a(n)\biggr|\le \sum_{1\le n\le N}|a(n)|\ll N^{\frac{k}{2}+1}.
\]
Thus, Corollary \ref{corollary7.4} saves roughly a factor $N^{1/p}$ over the trivial estimate, and as $p$ approaches 
$1$ the conclusion of Corollary \ref{corollary7.4} approaches the strength of the bound \eqref{7.1} applicable to 
unrestricted sums.\par

This circle of ideas also addresses character sums restricted to subconvex $L^p$-sets. We provide an illustration of 
the possible conclusions in this direction.

\begin{corollary}\label{corollary7.5}
Let $\chi$ be a non-principal character modulo $\pi$, with $\pi$ a prime number. Suppose that $\mathscr A$ is a 
strongly subconvex $L^p$-set for some real number $p$ with $1\le p<2$. Then, whenever $r$ is a positive integer 
exceeding $1$, one has
\[
\sum_{n\in \mathscr A(N)}\chi(n)\ll 
N\left( \pi^{\frac{1}{4r-4}}(\log \pi)^{\frac{1}{2r}}N^{-\frac{1}{r}}\right)^{\frac{2}{p}-1} .
\]
\end{corollary}

\begin{proof} We apply Theorem \ref{theorem7.1} with $f(n)=\chi(n)$. Here, as a consequence of Chamizo 
\cite[Theorem 1.1]{Cha2011}, one has
\[
\| H_f\|_\infty = \sup_{\beta\in [0,1)}\biggl| \frac{1}{N}\sum_{1\le n\le N}\chi(n)e(\beta n)\biggr| 
\ll N^{-\frac{1}{r}}\pi^{\frac{1}{4r-4}}(\log \pi)^{\frac{1}{2r}},
\]
whilst $\|f\|_{l^2(N)}\le 1$. Consequently, we find that the desired conclusion is immediate from Theorem 
\ref{theorem7.1}.
\end{proof}

Notice that the conclusion of Corollary \ref{corollary7.5} confirms that non-trivial cancellation occurs in the 
character sum
\[
\sum_{n\in \mathscr A(N)}\chi(n)
\]
whenever $N>\pi^{\frac{r}{4r-4}}(\log \pi)^{\frac{1}{2}}$. Thus, whenever $\varepsilon>0$, we may take $r$ 
sufficiently large in terms of $\varepsilon$ so as to ensure that non-trivial cancellation occurs provided only that 
$N>\pi^{1/4+\varepsilon}$.\par

We should note that, although we have restricted ourselves in this section to problems involving strongly 
subconvex $L^p$-sets, this is not an essential hypothesis. Were we to modify our hypotheses throughout to 
consider weakly subconvex $L^p$-sets, then a self-evident adaptation of our methods would yield the same upper 
bounds in every conclusion, save that an additional factor of $N^\varepsilon$ should be inserted throughout. We 
note further that the examples described above provide just some of many possible applications of the ideas within 
the orbit of this memoir. Moreover, more sophisticated strategies are suggested when information concerning 
$4$-th and higher moments of $H_f(\alpha;N)$ are available, as is often the case. The basic message to be borne 
away from this discussion is that subconvex $L^p$-sets lend themselves well to problems involving averages of 
arithmetic functions.

\section{Subconvex $L^p$-sets modulo $q$}
In certain applications of subconvex $L^p$-sets in number theory, it may be convenient to work with discrete rather 
than continuous moments, As a model example of such a scenario, consider a natural number $q$ and the 
discrete moment
\begin{equation}\label{8.1}
S_p(N;\mathscr A;q)=\frac{1}{q}\sum_{a=1}^q\biggl| \sum_{n\in \mathscr A(N)}e(na/q)\biggr|^p.
\end{equation}
Our principal result in this section is the analogue of the mean value estimate \eqref{1.3} for strongly subconvex 
$L^p$-sets $\mathscr A$ provided by Theorem \ref{theorem1.6}.

\begin{proof}[The proof of Theorem \ref{theorem1.6}]
We adjust the notation \eqref{3.2} by writing
\begin{equation}\label{8.2}
f_N(\alpha)=\sum_{n\in \mathscr A(N)}e(n\alpha).
\end{equation}
Thus, when $1<p<2$, one has
\begin{equation}\label{8.3}
|f_N(\alpha)|^p=f_N(\alpha)^{p/2}f_N(-\alpha )^{p/2}.
\end{equation}
Equipped with this relation, we seek to convert the discrete mean value $S_p(N;\mathscr A;q)$ into a continuous 
analogue by applying the Sobolev-Gallagher inequality (see \cite[Lemma 1.1]{HLM1971}). Thus, whenever 
$a\in \mathbb Z$ and $q\in \mathbb N$, we have
\[
|f_N(a/q)|^p\le q\int_{-1/(2q)}^{1/(2q)}|f_N(\beta+a/q)|^p\d\beta +
\frac{1}{2}\int_{-1/(2q)}^{1/(2q)}\Bigl| \frac{{\rm d}\,}{{\rm d}\beta} |f_N(\beta+a/q)|^p\Bigr| \d\beta .
\]
When $f_N(\beta+a/q)\ne 0$, it follows from \eqref{8.3} that
\[
\frac{{\rm d}\,}{{\rm d}\beta} |f_N(\beta+a/q)|^p=\frac{p}{2}|f_N(\beta+a/q)|^p
\biggl( \frac{f_N'(\beta+a/q)}{f_N(\beta+a/q)}-\frac{f_N'(-\beta-a/q)}{f_N(-\beta-a/q)}\biggr) ,
\]
whence
\[
\Bigl| \frac{{\rm d}\,}{{\rm d}\beta} |f_N(\beta+a/q)|^p\Bigr| \le p|f_N(\beta+a/q)|^{p-1}|f'_N(\beta+a/q)|.
\]
Since $p>1$, we therefore deduce from \eqref{8.1} that
\begin{equation}\label{8.4}
S_p(N;\mathscr A;q)\le \int_0^1|f_N(\alpha)|^p\d\alpha +
\frac{p}{2q}\int_0^1|f_N(\alpha)|^{p-1}|f'_N(\alpha)|\d\alpha .
\end{equation}

\par We must now bound mean values involving $f'_N(\alpha)$ in terms of mean values involving $f_N(\alpha)$. 
For this purpose, we draw inspiration from the argument following \cite[equation (3.2)]{Woo2015}. We observe that
\[
f_N'(\alpha)=2\pi {\rm i}\sum_{n\in \mathscr A(N)}ne(n\alpha).
\]
Recall the definition \eqref{8.1}. Then, on writing $\mathscr A(N)=\{ a_1,a_2,\ldots ,a_M\}$, with 
\[
a_1<a_2<\ldots <a_M,
\]
we see that the sum on the right hand side here may be rewritten via partial summation in the form
\begin{align*}
\sum_{m=1}^Ma_me(a_m\alpha)&=a_1f_{a_1}(\alpha)+
\sum_{m=2}^Ma_m\left( f_{a_m}(\alpha)-f_{a_{m-1}}(\alpha)\right) \\
&=a_Mf_{a_M}(\alpha)-\sum_{m=1}^{M-1}(a_{m+1}-a_m)f_{a_m}(\alpha).
\end{align*}
We therefore see that for all real numbers $\alpha$, one has
\begin{align*}
\biggl| \sum_{m=1}^Ma_me(a_m\alpha)\biggr|&\le a_M|f_{a_M}(\alpha)|+\left( \max_{1\le w\le M-1}|f_{a_w}(\alpha)|
\right) \sum_{1\le m\le M-1}|a_{m+1}-a_m|\\
&\le (a_M+(a_M-a_1))\max_{1\le w\le M}|f_{a_m}(\alpha)|\\
&\le 2N\max_{1\le n\le N}|f_n(\alpha)|.
\end{align*}
An application of the Carleson-Hunt theorem (see \cite[Theorem 1]{Hun1968}) in the form given by Bombieri 
\cite[page 12]{Bom1989} shows that
\[
\int_0^1\left( \max_{1\le n\le N}|f_n(\alpha)|\right)^p\d\alpha \ll_p \int_0^1|f_N(\alpha)|^p\d\alpha .
\]
We therefore deduce that whenever $\mathscr A$ is a strongly subconvex $L^p$-set, then one has
\begin{align}
\int_0^1|f'_N(\alpha)|^p\d\alpha &\ll N^p\int_0^1\left( \max_{1\le n\le N}|f_n(\alpha)|\right)^p\d\alpha \notag \\
&\ll N^p\int_0^1|f_N(\alpha)|^p\d\alpha \notag \\
&\ll N^{2p-1}.\label{8.5}
\end{align}

\par We now substitute the estimate \eqref{8.5} into \eqref{8.4} following an application of H\"older's inequality. 
Thus, on applying once again the hypothesis that $\mathscr A$ be a strongly subconvex $L^p$-set, we find that
\begin{align*}
S_p(N;\mathscr A;q)&\ll N^{p-1}+\frac{1}{q}\biggl( \int_0^1|f_N(\alpha)|^p\d\alpha\biggr)^{1-1/p}
\biggl( \int_0^1|f_N'(\alpha)|^p\d\alpha \biggr)^{1/p}\\
&\ll N^{p-1}+\frac{1}{q}(N^{p-1})^{1-1/p}(N^{2p-1})^{1/p}\\
&\ll N^{p-1}+N^p/q.
\end{align*}
This delivers the conclusion of the theorem.
\end{proof}  

In order to demonstrate the utility of this mean value estimate, we note an application to character sums that 
in many circumstances provides an alternative to the estimate given in Corollary \ref{corollary7.5}.

\begin{corollary}\label{corollary8.1}
Let $q$ be a large natural number, and suppose that $\chi$ is a primitive character modulo $q$. Let $p$ be a real 
number with $1<p<2$, and suppose that $\mathscr A$ is a strongly subconvex $L^p$-set. Then, one has
\[
\sum_{n\in \mathscr A(N)}\chi(n)\ll q^{1/2}N^{1-1/p}\left( 1+N/q\right)^{1/p}.
\]
\end{corollary}

\begin{proof} We recall the notation \eqref{8.2}, and in addition define
\[
T(a)=\sum_{m=1}^q\chi(m)e(ma/q).
\]
Then, by orthogonality, one sees that
\[
\sum_{n\in \mathscr A(N)}\chi(n)=\frac{1}{q}\sum_{a=1}^qT(a)f_N(-a/q).
\]
An application of H\"older's inequality reveals, therefore, that
\begin{equation}\label{8.6}
\biggl| \sum_{n\in \mathscr A(N)}\chi(n)\biggr| \le \biggl( \frac{1}{q}\sum_{a=1}^q|f_N(a/q)|^p\biggr)^{1/p}
\biggl( \frac{1}{q}\sum_{a=1}^q|T(a)|^{p/(p-1)}\biggr)^{1-1/p}.
\end{equation}
Since we assume $\chi$ to be primitive, it follows from \cite[Theorem 9.7]{MV2007} that $|T(a)|\le \sqrt{q}$. Thus, by 
orthogonality, we see that whenever $1<p<2$, one has
\[
\frac{1}{q}\sum_{a=1}^q|T(a)|^{\frac{p}{p-1}}\le (\sqrt{q})^{\frac{2-p}{p-1}}\frac{1}{q}\sum_{a=1}^q|T(a)|^2
\le q^{\frac{p}{2p-2}}.
\]
We therefore deduce from Theorem \ref{theorem1.6} and \eqref{8.6} that
\begin{align*}
\biggl| \sum_{n\in \mathscr A(N)}\chi(n)\biggr| &\ll \left( N^{p-1}(1+N/q)\right)^{1/p}
\left( q^{p/(2p-2)}\right)^{1-1/p}\\
&\ll q^{1/2}N^{1-1/p}(1+N/q)^{1/p}.
\end{align*}
This yields the conclusion of the corollary.
\end{proof}

Notice that as $p\rightarrow 1$, the estimate provided by Corollary \ref{corollary8.1} has strength approaching the 
classical estimate of P\'olya-Vinogradov (see \cite[Theorem 9.18]{MV2007}, for example). Meanwhile, our estimate 
remains non-trivial for all values of $p$ smaller than $2$.

\bibliographystyle{amsbracket}
\providecommand{\bysame}{\leavevmode\hbox to3em{\hrulefill}\thinspace}

\end{document}